\documentclass[11pt]{article}
\usepackage{amsmath}
\usepackage[thmmarks,amsmath,amsthm]{ntheorem}

\usepackage{amsfonts, psfrag, graphicx, amssymb,enumitem,indentfirst,float,geometry,color,enumerate,graphicx,subfigure,algorithm,algpseudocode,pifont,hyperref}

\allowdisplaybreaks[4]
\numberwithin{equation}{section}
\numberwithin{figure}{section}

\usepackage{lipsum}
\newcommand\blfootnote[1]{%
  \begingroup
  \renewcommand\thefootnote{}\footnote{#1}%
  \addtocounter{footnote}{-1}%
  \endgroup
}

\newtheorem{thm}{Theorem}[section]
\newtheorem{lemma}{Lemma}[section]
\newtheorem{rem}{Remark}[section]

\newcommand{\commentout}[1]{{}} 

\newcommand{\abs}[1]{\left|#1\right|}
\newcommand{\norm}[1]{\left\|#1\right\|}

\newcommand{\bfb}{{\bf b}}

\newcommand{\bfc}{{\bf c}}

\newcommand{\bfn}{{\bf n}}

\newcommand{\bfdelta}{\boldsymbol{\delta}}

\newcommand{\bfgamma}{\boldsymbol{\gamma}}

\newcommand{\bfLambda}{\boldsymbol{\Lambda}}

\begin{document}
\title{A Group of Immersed Finite Element Spaces \\
For Elliptic Interface Problems \thanks{This research was partially supported by the National Science Foundation (Grant Number DMS 1016313)}
}
\author{Ruchi Guo and Tao Lin\\
(ruchi91@vt.edu and tlin@vt.edu) \\
Department of Mathematics \\
Virginia Tech }
\date{}
\maketitle

\begin{abstract}
We present a unified framework for developing and analyzing immersed finite element (IFE) spaces for solving typical elliptic interface problems
 with interface independent meshes. This framework allows us to construct a group of new IFE spaces with either linear, or bilinear, or the rotated-$Q_1$ polynomials. Functions in these IFE spaces are locally piecewise polynomials defined according to the sub-elements formed by the interface itself instead of its line approximation. We show that the unisolvence for these IFE spaces follows from the invertibility of the Sherman-Morrison matrix. A group of estimates and identities are established for the interface geometry and shape functions that are applicable to all of these IFE spaces. These fundamental preparations enable us to develop a unified multipoint Taylor expansion procedure for proving that these IFE spaces have the expected optimal approximation capability according to the involved polynomials.
\end{abstract}


\section{Introduction}

This article presents a unified framework for developing and analyzing a group of immersed finite element (IFE) spaces that use interface independent meshes (such as highly structured \textrm{Cartesian} meshes) to solve interface problems of the typical second order elliptic partial differential equations:
\begin{align}
\label{eq1_1}
 -\nabla\cdot(\beta\nabla u)=f, \;\;\;\; & \text{in} \; \Omega^-  \cup \Omega^+, \\
 u=g, \;\;\;\; &\text{on} \; \partial\Omega,
\end{align}
where, without loss of generality, the domain $\Omega\subseteq \mathbb{R}^2$ is separated by an interface curve $\Gamma$ into two subdomains $\Omega^+$ and $\Omega^-$, the diffusion coefficient $\beta(X)$ is discontinuous such that
\begin{equation*}
\beta(X)=
\left\{\begin{array}{cc}
\beta^- & \text{if} \; X\in \Omega^- ,\\
\beta^+ & \text{if} \; X\in \Omega^+ ,
\end{array}\right.
\end{equation*}
where $\beta^{\pm}$ are positive constants. In addition, the solution $u$ is assumed to satisfy the jump conditions:
\begin{align}
[u]_{\Gamma} &= 0, \label{eq1_3} \\
\big[\beta \nabla u\cdot \mathbf{n}\big]_{\Gamma} &= 0, \label{eq1_4}
\end{align}
where $\mathbf{ n}$ is the unit normal vector to the interface $\Gamma$, and for every piece-wise function $v$ defined as
\begin{equation*}
v=
\left\{\begin{array}{cc}
v^-(X) & \text{if} \; X\in \Omega^- , \\
v^+(X) & \text{if} \; X\in \Omega^+ ,
\end{array}\right.
\end{equation*}
we adopt the notation $[v]|_{\Gamma}=v^+|_{\Gamma}-v^-|_{\Gamma}$.

\par
It is well known that the standard finite element method can be applied to interface problems provided that the mesh is formed according to the interface, see \cite{1970Babuska,1996BrambleKing,1998ChenZou} and references therein. Many efforts have been made to develop alternative finite element
methods based on unfitted meshes for solving interface problems. The advantages of using unfitted meshes are discussed in
\cite{2010Bordas,2014MoumnassiFigueredo,2011MohammedSalim,2013NadalAlbelda}. A variety of finite element methods that can use interface independent meshes to solve interface problems have been reported in the literature, see
\cite{2003BabuskaBanerjeeOsborn,
1998BabuskaZhang,
1987BarrettElliott,
2009BastianEngwer,
2001DolbowMoesBelytschko,2009EfendievHou,2009GuyomarchLeeJeon,2015GuzmanSanchezSarkisMathComp,1996BabuskaMelenk} for a few examples. In particular,
instead of modifying the shape functions on interface elements which is an approach to be discussed in this article, methods in \cite{2015BurmanClaus,2016BurmanGuzmanSanchezSarkisHcontrast,2002HansboHansbo,2016WangXiaoXu} employ standard finite element functions but use the Nitsche's penalty along the interface in the finite element schemes.

\commentout{
{\color{red}In the literature, alternative finite element methods have been developed for using unfitted meshes. For example, in stead of modifying the shape functions on interface elements, methods in \cite{2015BurmanClaus,2016BurmanGuzmanSanchezSarkisHcontrast,2002HansboHansbo,2016WangXiaoXu}
employ the Nitsche's penalty along the interface in the finite element schemes. Other finite element methods that can use interface independent meshes to solve interface problems can be found in
\cite{2003BabuskaBanerjeeOsborn,
1998BabuskaZhang,
1987BarrettElliott,
2009BastianEngwer,
2001DolbowMoesBelytschko,2009EfendievHou,2009GuyomarchLeeJeon,2015GuzmanSanchezSarkisMathComp,1996BabuskaMelenk}. The unfitted mesh is also advantageous when the boundary geometry of the computation is complicated or moving during the computation \cite{2014MoumnassiFigueredo,2011MohammedSalim,2013NadalAlbelda}. In addition, the authors in \cite{2010Bordas} review different approaches aiming at reducing the mesh burden in computation.}
}

This article focuses on the immersed finite element (IFE) methods, whose basic idea was introduced in \cite{1998Li}, for those applications where it is preferable to solve interface problems with a mesh independent of the interface,
for example, the Particle-In-Cell method for plasma particle simulations \cite{2005KafafyWangLin,2007KafafyWang,2002LinWang}, the problems with moving interfaces
\cite{2013HeLinLinZhang,2013LinLinZhang1}, and the electroencephalography forward problem \cite{2010VallaghePapadopoulo}. IFE methods for interface problems of other types of partial differential equations can be found in \cite{2015AdjeridChaabaneLin,2013HeLinLinZhang,2012HouLiWangWang,2005LiYang,2011LinLinSunWang,2013LinLinZhang1,2013LinSheenZhang,2016KihyoMoon_thesis,2003YangLiLi}.

The IFE spaces developed in this article are extended from the IFE spaces constructed with linear polynomials \cite{2007GongLiLi,2010KwakWeeChang,2004LiLinLinRogers,2003LiLinWu}, bilinear polynomials
\cite{2009HeTHESIS,2008HeLinLin,2001LinLinRogersRyan}, and rotate-$Q_1$ polynomials \cite{2013LinSheenZhang,2013ZhangTHESIS} using the standard Lagrange type local degrees of freedom imposed either
at element vertices as usual or at midpoints of element edges in the Crouzeix-Raviart way \cite{1973CrouzeixRaviart}. We note that, the local linear IFE space on each triangular interface element constructed here with the Lagrange local degrees of freedom imposed at vertices is very similar to the one recently introduced in \cite{2015GuzmanSanchezSarkisP1}.
The IFE spaces in this article are new because, locally on each interface element, they are Hsieh-Clough-Tocher type macro finite element functions \cite{2001Braess,1966CloughTocher} defined with sub-elements formed by the interface curve itself in contrast to those IFE spaces in the literature defined with sub-elements formed by a straight line approximating the interface curve.

Our research presented here is motivated by two issues. The first issue concerns the general $O(h^2)$ order accuracy for a line to approximate a curve which is a fundamental ingredient for the optimal approximation capability of those IFE spaces in the literature. We hope the study of IFE spaces based on curve sub-elements can shed light on the development of higher degree IFE spaces for which the $O(h^2)$ order is not sufficient. For examples, those techniques incorporating the exact geometry for constructing basis functions \cite{2012LianBordas,2015NguyenAnitescu,2008SevillaSoniaHuerta} may be considered. The second issue is the attempt to unify the fragmented framework for developing and analyzing the IFE spaces in the literature. For IFE spaces based on different meshes, different polynomials, and different local degrees of freedom, we show that their unisolvence, i.e.,
the existence and uniqueness of IFE shape functions, can be established through a uniform procedure related with the invertibility of the Sherman-Morrison matrix. We have derived a group of identities for the interface geometry and shape functions that are applicable to all of these IFE spaces, and this enables us to derive error estimates for the interpolation in these new IFE spaces in a general unified multipoint Taylor expansion approach in which, IFE functions defined according to the given interface actually simplify the analysis because we only need to apply the same arguments to two sub-elements formed by the interface while the analysis for the IFE spaces in the literature has to use a different set of arguments to handle the sub-elements sandwiched between the interface curve and its approximate line.
Also, inspired by \cite{2015GuzmanSanchezSarkisP1}, we have made an effort to show how the error bounds explicitly depend on the maximum curvature of the interface curve and the ratio between $\beta^-$ and $\beta^+$, which are two important problem dependent characteristics effecting the approximation capability of IFE spaces. We note that the dependence of constants in the error bounds on the ratio between $\beta^-$ and $\beta^+$ in our article is similar to the one discussed in \cite{2010ChuGrahamHou}, and we think this coincidence follows from the fact that we analyze the approximation capability of
finite element spaces with Lagrange type degrees of freedom.


This article consists of 5 additional sections. In the next section we describe common notations and some basic assumptions used in this article. In Section 3, we derive estimates and identities associated with the interface and the jump conditions in an element. From these estimates, we can see how their bounds explicitly depend on
curvature of the interface and the ratio between $\beta^-$ and $\beta^+$, and how the mesh size $h$ is subject to the interface curvature.
In Section 4 we present generalized multipoint Taylor expansions for piecewise $C^2$ functions in an interface element. Estimates for the remainders in these expansions are derived in terms of pertinent Sobolev norms. In Section 5, first, we establish the unisolvence of immersed finite element functions constructed with linear, bilinear, Crouzeix-Raviart and rotated-$Q_1$ polynomials, i.e., we show the standard Lagrange local degrees of freedom imposed at the nodes of an interface element can uniquely determine an IFE function that satisfies the interface jump conditions in a suitable approximate sense. Then, we show that the IFE shape functions have several desirable properties such as the partition of unity and the critical identities in Theorem \ref{thm4_3}.
Finally, with a unified analysis, we show that the IFE spaces have the expected optimal approximation capability. In Section 6, we demonstrate features of these IFE spaces by numerical examples.

\section{Preliminaries}
Throughout the article, $\Omega\subset\mathbb{R}^2$ denotes a bounded domain as a union of finitely many rectangles. The interface curve $\Gamma$ separates $\Omega$ into two subdomains $\Omega^+$ and $\Omega^-$ such that $\overline{\Omega} = \overline{\Omega^+} \cup \overline{\Omega^-} \cup \Gamma$. For every measurable subset $\tilde \Omega \subseteq \Omega$, let $W^{k,p}(\tilde \Omega)$ be the standard Sobolev spaces on $\tilde \Omega$ associated with the norm $\|\cdot\|_{k,p,\tilde \Omega}$ and the semi-norm $|v|_{k,p,\tilde \Omega}=\|D^{\alpha}v\|_{0,p,\tilde \Omega}$, for $|\alpha|=k$. The corresponding Hilbert space is $H^k(\tilde \Omega)=W^{k,2}(\tilde \Omega)$. When
$\tilde \Omega^s = \tilde \Omega \cap \Omega^s \not = \emptyset, s = \pm$, we let
\begin{equation*}
PH^k_{int}(\tilde \Omega)=\{ u: u|_{\tilde{\Omega}^s}\in H^k(\tilde \Omega^s),\; s=\pm; \; [u]_\Gamma=0 \; \text{and} \; [\beta\nabla u\cdot \mathbf{n}]_\Gamma=0\; \text{on}\; \Gamma \cap \tilde \Omega \},
\end{equation*}
\begin{equation*}
PC^k_{int}(\tilde \Omega)=\{ u: u|_{\tilde \Omega^s}\in C^k(\tilde \Omega^s),\; s=\pm; \; [u]_\Gamma=0 \; \text{and} \; [\beta\nabla u\cdot \mathbf{n}]_{\Gamma}=0\; \text{on}\; \Gamma \cap \tilde \Omega \}.
\end{equation*}
The norms and semi-norms to be used on $PH^k_{int}(\tilde \Omega)$ are
\begin{equation*}
\|\cdot\|^2_{k,\tilde \Omega}=\|\cdot\|^2_{k,\tilde \Omega^+}+\|\cdot\|^2_{k,\tilde \Omega^-}, \;\;\;\;\; |\cdot|^2_{k,\tilde \Omega}=|\cdot|^2_{k,\tilde \Omega^+}+|\cdot|^2_{k,\tilde \Omega^-},
\end{equation*}
\begin{equation*}
\|\cdot\|_{k,\infty,\tilde \Omega}=\max(\|\cdot\|_{k,\infty,\tilde \Omega^+} \;,\; \|\cdot\|_{k,\infty,\tilde \Omega^-}), \;\;\;\;\; |\cdot|_{k,\infty,\tilde \Omega}=\max(|\cdot|_{k,\infty,\tilde \Omega^+} \;,\; |\cdot|_{k,\infty,\tilde \Omega^-}).
\end{equation*}

Let $\mathcal{T}_h$ be a Cartesian triangular or rectangular mesh of the domain $\Omega$ with the maximum length of edge $h$. An element $T\in \mathcal{T}_h$ is called an interface element provided the interior of $T$ intersects with the interface $\Gamma$; otherwise, we name it a non-interface element. We let $\mathcal{T}^i_h$ and $\mathcal{T}^n_h$ be the set of interface elements and non-interface elements, respectively. Similarly, $\mathcal{E}^i_h$ and $\mathcal{E}^n_h$ are sets of interface edges and non-interface edges, respectively. In addition, as in \cite{2009HeLinLin}, we assume that $\mathcal{T}_h$ satisfies the following hypotheses when the mesh size $h$ is small enough:

\begin{itemize}[leftmargin=30pt]
  \item [\textbf{(H1)}] The interface $\Gamma$ cannot intersect an edge of any element at more than two points unless the edge is part of $\Gamma$.
  \item [\textbf{(H2)}] If $\Gamma$ intersects the boundary of an element at two points, these intersection points must be on different edges of this element.
  \item [\textbf{(H3)}] The interface $\Gamma$ is a piecewise $C^2$ function, and the mesh $\mathcal{T}_h$ is formed such that the subset of $\Gamma$ in every interface element $T\in\mathcal{T}^i_h$ is $C^2$.
  \item [\textbf{(H4)}] The interface $\Gamma$ is smooth enough so that $PC^2_{int}(T)$ is dense in $PH^2_{int}(T)$ for every interface element $T\in\mathcal{T}^i_h$.
\end{itemize}

On an element $T \in \mathcal{T}_h$, we consider the local finite element space $(T, \Pi_T, \Sigma_T^P)$ with
\begin{eqnarray}
\Pi_T &=& \begin{cases}
\textrm{Span} \{ 1,x,y\}, & \text{for $P_1$ or Crouzeix-Raviart (C-R) finite element functions}, \\
\textrm{Span} \{ 1,x,y,xy\}, & \text{for $Q_1$ (bilinear) finite element functions}, \\
\textrm{Span} \{ 1,x,y,x^2-y^2 \}, & \text{for rotated-$Q_1$ finite element functions},
\end{cases} \label{eq4_1}\\
\Sigma_T^P &=& \{ \psi^P_T(M_i): i \in \mathcal{I}, \; \forall \psi^P_T \in \Pi_T \}, \label{eq4_2}
\end{eqnarray}
where $\mathcal{I} = \{1, 2, \cdots, DOF(T)\}$, $DOF(T) = 3$ or $4$ depending on whether $T$ is triangular or rectangular, $M_i, i\in \mathcal{I}$ are the local nodes to determine
shape functions on $T$, and the super script $P$ is to emphasize the Lagrange type degrees of freedom imposed at the points $M_i$s. For $P_1$ and $Q_1$ finite elements, $M_i = A_i, i\in \mathcal{I}$ where $A_i$'s are vertices of $T$. For C-R and rotated-$Q_1$ finite elements,
$M_i$ is the midpoint of the $i$-th edge of $T$ for $i\in \mathcal{I}$. It is well known \cite{1994BrennerScott,1978Ciarlet, 1973CrouzeixRaviart,1992RanacherTurek} that
$(T, \Pi_T, \Sigma_T^P)$ has a set of shape functions $\psi_i^P(X), i\in \mathcal{I}$ such that
\begin{equation}
\label{eqn_psi_iT_values_bounds}
\psi^P_{i,T}(M_j)=\delta_{ij},~~\norm{\psi^P_{i,T}}_{\infty, T} \leq C,~~\norm{\nabla \psi^P_{i,T}}_{\infty, T} \leq Ch^{-1}, \;\; i,j\in \mathcal{I},
\end{equation}
where $\delta_{ij}$ is the \textit{Kronecker} delta function.

Throughout this article, without loss of generality, we assume that $\beta^+\geqslant\beta^-$ and let $\rho=\beta^-/\beta^+\leqslant1$. In addition, on any $T\in \mathcal{T}^i_h$, we use $D$, $E$ to denote the intersection points of $\Gamma$ and $\partial T$ and let $l$ be the line connecting $DE$.


\section{Geometric Properties of the Interface}

In this section, we discuss geometric properties on interface elements that are useful for developing and analyzing IFE spaces. Let $T$ be an interface element.
As illustrated in Figure \ref{fig2_2}, for a point $\widetilde{X}$ on $\Gamma\cap T$, we let
$\mathbf{n}(\widetilde{X})=(\tilde{n}_x(\widetilde{X}), \tilde{n}_y(\widetilde{X}))$ be the normal of $\Gamma$ at $\widetilde{X}$, and we let $\widetilde{X}_{\bot}\in l$ be the orthogonal projection of $\widetilde{X} \in \Gamma\cap T$ onto the line $l$. Also, for the line $l$, we let $\bar{\mathbf{n}}=(\bar{n}_x,\bar{n}_y)$ be its unit normal vector and, consequently, $\bar{\mathbf{t}}=(\bar{n}_y, -\bar{n}_x)$ is the vector tangential to $l$. Without loss of generality, we assume the orientation of all the normal vectors are from $T^-$ to $T^+$. In addition, we let $\kappa$ be the maximum curvature of the curve $\Gamma$.

\begin{figure}[H]
\centering
\begin{minipage}[t]{0.4\textwidth}
\centering
\includegraphics[width=2.5in]{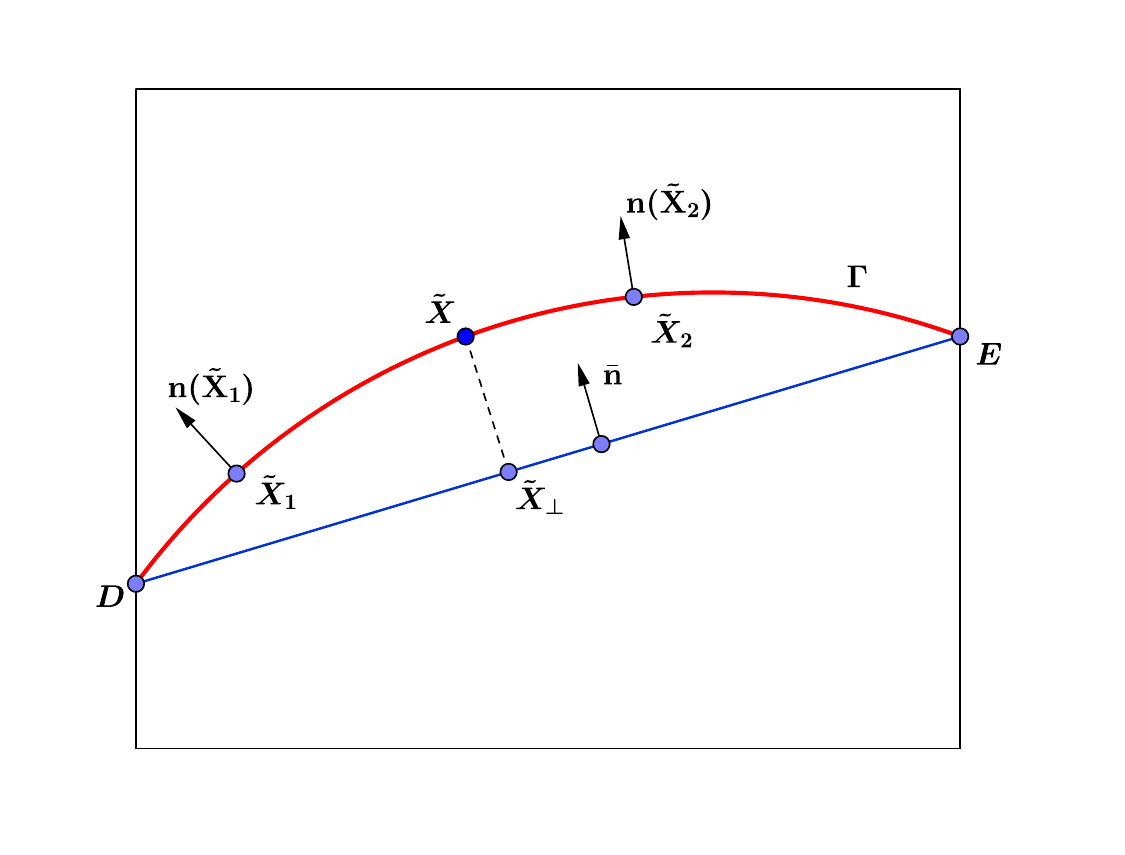}
\caption{ The local system }
\label{fig2_2}
\end{minipage}
\begin{minipage}[t]{0.5\textwidth}
\centering
\includegraphics[width=2.5in]{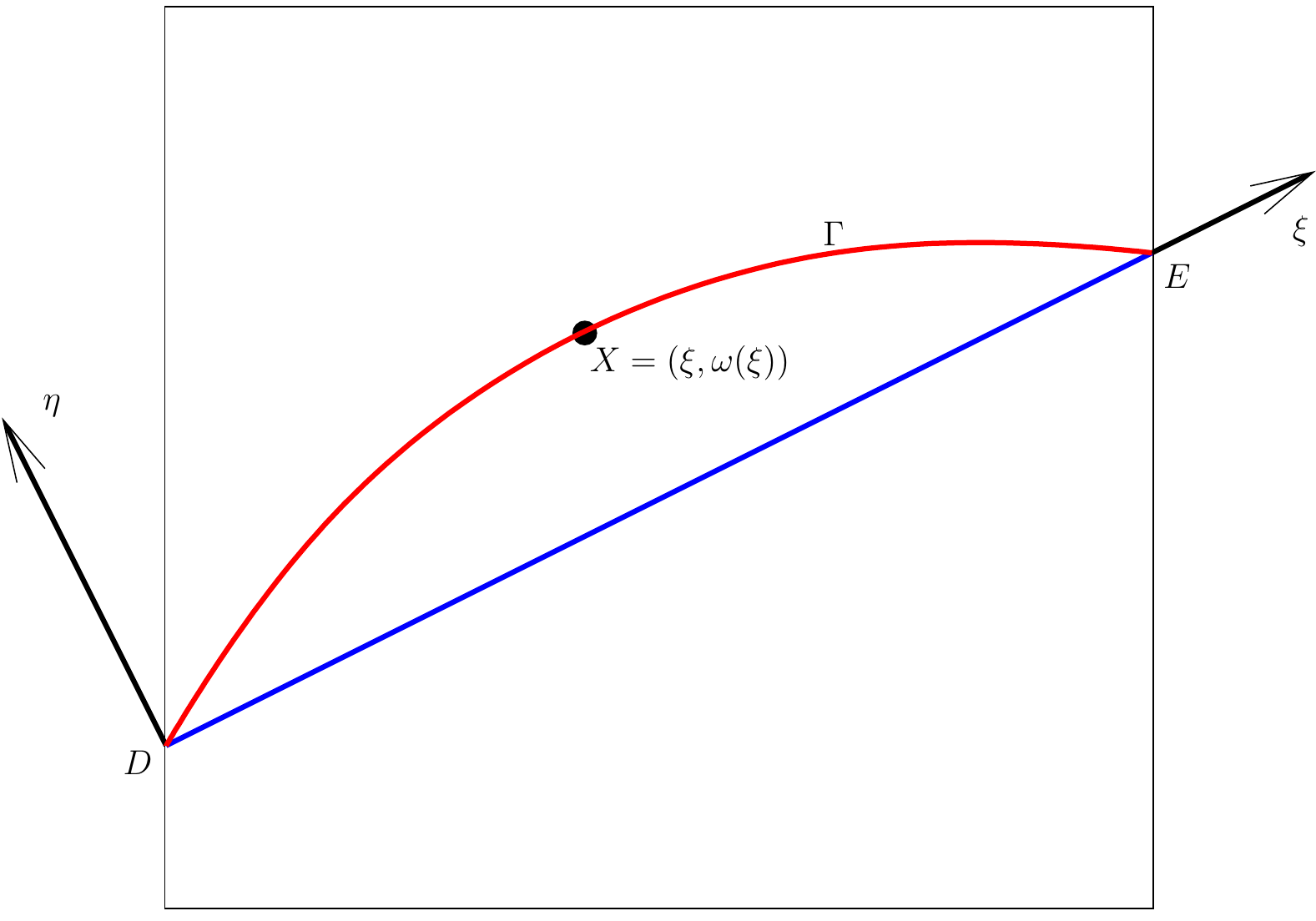}
\caption{The geometry of an interface element }
\label{fig2_1}
\end{minipage}
\end{figure}

For $T \in \mathcal{T}_h^i$, without loss of generality, we can introduce a local coordinate system such that
the point on $\Gamma \cap T$ can expressed as $(\xi,\eta) = (\xi, \omega(\xi))$ for a suitable function $\omega$ such that, in this local system, $D$ is origin and its $x$-axis is aligned with $l$, as shown in Figure \ref{fig2_1}. We start from the following lemma which extends
similar results in \cite{2004LiLinLinRogers}:

\begin{lemma}
\label{lem_geo_curv}
Given any $\epsilon\in(0,\frac{\sqrt{2}}{2})$, assume $h\kappa\leqslant\epsilon$, then for any interface element $T\in\mathcal{T}^i_h$, there hold
\begin{eqnarray}
&&|w(\xi)|\leqslant 2(1-2\epsilon^2)^{-3/2}\kappa h^2, \label{eq2_1} \\
&&|w'(\xi)|\leqslant \sqrt{2}(1-2\epsilon^2)^{-3/2}\kappa h. \label{eq2_2}
\end{eqnarray}
\end{lemma}
\begin{proof}
In the local system, let $\xi_E$ be the coordinate of the point $E$. And by the Mean Value Theorem, there is some $\xi_0\in[0,\xi_E]$ such that $\omega'(\xi_0)=0$. Consider a function $g$ as well as its derivative
$$
g(\xi)=\frac{\omega'(\xi)}{\sqrt{1+(\omega'(\xi))^2}}, ~~~ \textrm{and}, ~~~ g'(\xi)=\frac{\omega''(\xi)}{(1+(\omega'(\xi))^2)^{3/2}}.
$$
Note that $g'(\xi)$ is the curvature of $\Gamma$ at $\xi$; hence, we have $|g'(\xi)| \leq \kappa$. Then, by $g(\xi_0)=0$, we have $|g(\xi)|=\abs{\int^{\xi}_{\xi_0}g'(s)ds}\leqslant\int^{\xi}_{\xi_0}|g'(s)|ds\leqslant\sqrt{2}\kappa h=\sqrt{2}\epsilon$, which implies $|\omega'(\xi)|\leqslant\frac{\sqrt{2}\epsilon}{\sqrt{1-2\epsilon^2}}$. By the definition of $\kappa$, we have $|\omega''(\xi)|\leqslant (1-2\epsilon^2)^{-3/2} \kappa$.

Now using the Taylor expansion for $\omega$ around $D$ leads to
$
\omega(\xi)=\omega'(0)\xi+\frac{1}{2}\omega''(\bar{\xi})\xi^2
$
 for some $\bar{\xi}\in[0,\xi_E]$. Note that $\omega(\xi_E)=0$ shows $\omega'(0)=-\frac{1}{2}\omega''(\bar{\xi}_E))\xi_E$. Thus we have $\omega(\xi)=-\frac{1}{2}\omega''(\bar{\xi}_E))\xi_E\xi+\frac{1}{2}\omega''(\bar{\xi})\xi^2$ and therefore, $|\omega(\xi)|\leqslant 2\|\omega''\|_{\infty}h^2\leqslant2(1-2\epsilon^2)^{-3/2}\kappa h^2$, which yields \eqref{eq2_1}. And using the Taylor expansion again for $\omega'$ around $\xi_0$, we have $\omega'(\xi)=\omega''(\tilde{\xi})(\xi-\xi_0)$ for some $\tilde{\xi}$ between $\xi_0$ and $\xi$. Finally we obtain $|\omega'(\xi)|\leqslant \sqrt{2}\|\omega''\|h\leqslant\sqrt{2}(1-2\epsilon^2)^{-3/2}h$.
\end{proof}
%

%

We note that the argument in the lemma above is similar to Assumption 3.14 in \cite{2010ChuGrahamHou} with a minor difference that
a local polar coordinate system is used on the interface element in \cite{2010ChuGrahamHou}.
The following lemmas provide estimates about various geometric quantities defined at points on $\Gamma \cap T$.

\begin{lemma}
\label{lem2_3}
Given any $\epsilon\in(0,\frac{\sqrt{2}}{2})$, assume $h\kappa\leqslant\epsilon$, then for any interface element $T\in\mathcal{T}^i_h$ and any point $\widetilde{X}\in\Gamma\cap T$, the following inequality holds:
\begin{equation}
\label{eq2_6}
    \| \widetilde{X}-\widetilde{X}_{\bot}\| \leqslant 2(1-2\epsilon^2)^{-3/2}\kappa h^2,
\end{equation}
and for any $\widetilde{X}_1$, $\widetilde{X}_2 \in \Gamma \cap T$, we have
\begin{subequations} \label{eq2_7}
\begin{align}
    & \| \mathbf{ n}(\widetilde{X}_1)-\mathbf{ n}(\widetilde{X}_2)\|\leqslant \sqrt{2}(1+(1-2\epsilon^2)^{-3/2})\kappa h,  \label{eq2_7a} \\
    & \mathbf{ n}(\widetilde{X}_1)\cdot\mathbf{ n}(\widetilde{X}_2) \geqslant 1- 2(1+(1-2\epsilon^2)^{-3/2})^2\kappa^2h^2. \label{eq2_7b}
\end{align}
\end{subequations}
\end{lemma}
\begin{proof}
Estimate \eqref{eq2_6} directly follows from \eqref{eq2_1}. For \eqref{eq2_7a}, we assume $\widetilde{X}_1 = (\xi_1, w(\xi_2))$ and $\widetilde{X}_2 = (\xi_2, w(\xi_2))$ in the local system, respectively. Then we have
\begin{equation*}
\mathbf{ n}(\widetilde{X}_1)=\frac{1}{\sqrt{1+(w'(\xi_1))^2}}
\left(\begin{array}{c}
-w'(\xi_1) \\
1
\end{array}\right),
 \;\;\;
\mathbf{ n}(\widetilde{X}_2)=\frac{1}{\sqrt{1+(w'(\xi_2))^2}}
\left(\begin{array}{c}
-w'(\xi_2) \\
1
\end{array}\right).
\end{equation*}
By the calculation in Lemma \ref{lem_geo_curv} and Mean Value Theorem, there is some $\bar{\xi}\in[0,\xi_E]$ such that
\begin{equation*}
\begin{split}
\abs{\frac{w'(\xi_1)}{\sqrt{1+(w'(\xi_1))^2}} - \frac{w'(\xi_2)}{\sqrt{1+(w'(\xi_2))^2}}}
=\frac{|\omega''(\bar{\xi})|}{(1+(\omega'(\bar{\xi}))^{2})^{3/2}}|\xi_1-\xi_2|\leqslant \sqrt{2}\kappa h,
\end{split}
\end{equation*}
and $\tilde{\xi}\in[0,\xi_E]$ such that
\begin{equation*}
\begin{split}
\abs{\frac{1}{\sqrt{1+(w'(\xi_1))^2}} - \frac{1}{\sqrt{1+(w'(\xi_2))^2}}} = \frac{|\omega'(\tilde{\xi})| |\omega''(\tilde{\xi})|}{ (1+(\omega'(\tilde{\xi}))^{2})^{3/2} } |\xi_1-\xi_2| 
\leqslant 2(1-2\epsilon^2)^{-3/2}\epsilon \kappa h.
\end{split}
\end{equation*}
Then \eqref{eq2_7a} follows by applying these estimates in the local coordinate forms of $\mathbf{ n}(\widetilde{X}_1)$ and $\mathbf{ n}(\widetilde{X}_2)$.
Furthermore, by \eqref{eq2_7a} and
\begin{equation*}
\begin{split}
\| \mathbf{ n}(\widetilde{X}_1)-\mathbf{ n}(\widetilde{X}_2)\|^2
=\| \mathbf{ n}(\widetilde{X}_1)\|^2+\|\mathbf{ n}(\widetilde{X}_2)\|^2-2\mathbf{ n}(\widetilde{X}_1)\cdot\mathbf{ n}(\widetilde{X}_2)
=2-2\mathbf{ n}(\widetilde{X}_1)\cdot\mathbf{ n}(\widetilde{X}_2),
\end{split}
\end{equation*}
we have \eqref{eq2_7b}.
\end{proof}

\begin{rem}\label{rem2_1}
Note that there exists a point $\widetilde{X}_1 = (\xi_1, w(\xi_1))\in \Gamma\cap T$ such that $w'(\xi_1)=0$ which means $\mathbf{n}(\widetilde{X}_1) = \bar{\mathbf{n}}$. Hence, by Lemma \ref{lem2_3}, we have the following estimates for an arbitrary point $\widetilde{X} \in \Gamma \cap T$:
\begin{subequations}
\begin{align}
    & \| \mathbf{ n}(\widetilde{X})-\bar{\mathbf{ n}}\|\leqslant  \sqrt{2}(1+(1-2\epsilon^2)^{-3/2})\kappa h,  \label{eq2_8a}\\
    & \mathbf{ n}(\widetilde{X})\cdot \bar{\mathbf{ n}} \geqslant 1- 2(1+(1-2\epsilon^2)^{-3/2})^2\kappa^2 h^2. \label{eq2_8b}
\end{align}
\end{subequations}
\end{rem}

The two lemmas above have suggested a criteria about how small $h$ should be according to the maximum curvature $\kappa$ of $\Gamma$ so that the related analysis is valid.
Therefore, for all discussions from now on, we further assume that
\begin{itemize}[leftmargin=30pt]
\item
$h$ is sufficiently small such that for some fixed parameter $\epsilon\in (0,\sqrt{2}/2)$ and $\bar{\kappa} \in (0, 1]$ of one's own choice, there holds
\begin{equation}
\label{h_small_assump}
h<\min{ \left\{ \frac{\sqrt{\bar{\kappa}}}{\sqrt{2}(1+(1-2\epsilon^2)^{-3/2})\kappa}, \; \frac{\epsilon}{\kappa} \right\} }.
\end{equation}
\end{itemize}
Obviously $\epsilon$ is the proportion by which we should choose the mesh size $h$ according to the interface curvature $\kappa$. Also, by \eqref{h_small_assump}
and \eqref{eq2_8b}, we have
\begin{eqnarray}
\mathbf{ n}(\widetilde{X})\cdot \bar{\mathbf{ n}} \geq 1 - \bar{\kappa} \label{n_nbar_angle_lower_bnd}
\end{eqnarray}
which shows how much the angle between the normal of $\Gamma\cap T$ and $\bar{\mathbf{ n}}$ can vary in an interface element $T \in \mathcal{T}^i_h$, a larger value of
$\bar{\kappa}\in (0, 1]$ allows $\mathbf{ n}(\widetilde{X})$ to vary more from $\bar{\mathbf{ n}}$ up to, but not equal to, $90$ degree. Therefore, we will call $\bar{\kappa}$ the angle allowance.

In the rest of this article, all the generic constants $C$ are assumed to possibly depend only on the parameter $\epsilon$ and $\bar{\kappa}$, but they are independent of the interface location, $\beta^{\pm}$, and the curvature $\kappa$.

We now consider some matrices associated with the normal of interface $\Gamma$ and the normal of $l$. First, for any $\widetilde{X}\in\Gamma\cap T$, we use the normal $\mathbf{n}(\widetilde{X})$ to form two matrices:
\begin{equation*}
N^s(\widetilde{X})=\left(\begin{array}{cc}
\tilde{n}_y(\widetilde{X}) & -\tilde{n}_x(\widetilde{X}) \\
\beta^s \tilde{n}_x(\widetilde{X}) & \beta^s \tilde{n}_y(\widetilde{X})
\end{array}\right), ~~s = \pm.
\end{equation*}
Since $Det(N^s(\widetilde{X}))=\beta^s, s = \pm$, these matrices are nonsingular; therefore, we can define another two matrices at the
point $\widetilde{X}\in\Gamma\cap T$:
\begin{equation}
\label{eq2_3}
M^-(\widetilde{X})=\left(N^+(\widetilde{X})\right)^{-1}N^-(\widetilde{X})=
\left(\begin{array}{cc}
\tilde{n}^2_y(\widetilde{X})+\rho \tilde{n}^2_x(\widetilde{X}) & (\rho-1)\tilde{n}_x(\widetilde{X})\tilde{n}_y(\widetilde{X}) \\
(\rho-1)\tilde{n}_x(\widetilde{X})\tilde{n}_y(\widetilde{X}) & \tilde{n}^2_x(\widetilde{X})+\rho \tilde{n}^2_y(\widetilde{X})
\end{array}\right),
\end{equation}
\begin{equation}
\label{eq2_4}
M^+(\widetilde{X})=\left(N^-(\widetilde{X})\right)^{-1}N^+(\widetilde{X})=
\left(\begin{array}{cc}
\tilde{n}^2_y(\widetilde{X})+1/\rho \tilde{n}^2_x(\widetilde{X}) & (1/\rho-1)\tilde{n}_x(\widetilde{X})\tilde{n}_y(\widetilde{X}) \\
(1/\rho-1)\tilde{n}_x(\widetilde{X})\tilde{n}_y(\widetilde{X}) & \tilde{n}^2_x(\widetilde{X})+1/\rho \tilde{n}^2_y(\widetilde{X})
\end{array}\right).
\end{equation}
For matrices $M^-(\widetilde{X})$ and $M^+(\widetilde{X})$, we recall from \cite{2004LiLinLinRogers} the following results
\begin{equation}
\label{eq2_5}
\nabla u^+(\widetilde{X}) = M^-(\widetilde{X}) \nabla u^-(\widetilde{X}), \;\;\;\;\;
\nabla u^-(\widetilde{X})  = M^+(\widetilde{X}) \nabla u^+(\widetilde{X}), ~\forall \widetilde{X}\cap T \in \Gamma, ~~\forall u\in PC^2_{int}(T).
\end{equation}


In addition, for $\widetilde{X}\in\Gamma\cap T$, we can use the normal vectors $\mathbf{n}(\widetilde{X})$ and $\bar{\mathbf{n}}$ to form the following matrices:
\begin{equation*}
\overline{N}^s(\widetilde{X})=\left(\begin{array}{cc}
\bar{n}_y & -\bar{n}_x \\
\beta^s\tilde{n}_x(\widetilde{X}) & \beta^s\tilde{n}_y(\widetilde{X})
\end{array}\right), ~s = \pm.
\end{equation*}
By Remark \ref{rem2_1}, we have
\begin{equation*}
\begin{split}
Det(\overline{N}^{s}(\widetilde{X}))= \beta^{s}\mathbf{ n}(\widetilde{X})\cdot \bar{\mathbf{ n}}\geqslant \beta^{s} (1- \bar{\kappa}), \;\;\; s=\pm ,
\end{split}
\end{equation*}
which means $\overline{N}^{s}(\widetilde{X})$ are non-singular when $h$ is small enough; hence, we can use them to form
\begin{equation}
\label{eq2_9}
\overline{M}^+(\widetilde{X})=(\overline{N}^-(\widetilde{X}))^{-1}\overline{N}^+(\widetilde{X}), \;\;\;\; \overline{M}^-(\widetilde{X})=(\overline{N}^+(\widetilde{X}))^{-1}\overline{N}^-(\widetilde{X}).
\end{equation}


\begin{lemma}
\label{lem2_6}
For the mesh $\mathcal{T}_h$ with $h$ sufficiently small, there exists a constant $C$ independent of interface location, $\beta^{\pm}$, and $\kappa$, such that, for two arbitrary points $\widetilde{X}_i, i = 1, 2$ on $\Gamma \cap T$, we have
\begin{equation}
\label{bounds_for_Mbar}
\| \overline{M}^-(\widetilde{X}_1)\| \leqslant C, ~~~ \| \overline{M}^+(\widetilde{X}_1)\| \leqslant \frac{C}{\rho}, ~~~ \| M^-(\widetilde{X}_2)\| \leqslant C, ~~~ \| M^+(\widetilde{X}_2)\| \leqslant \frac{C}{\rho},
\end{equation}
and
\begin{equation}
\label{eq2_13}
\| \overline{M}^-(\widetilde{X}_1)-M^-(\widetilde{X}_2)\| \leqslant C\kappa h, ~~~~ \| \overline{M}^+(\widetilde{X}_1)-M^+(\widetilde{X}_2)\| \leqslant \frac{C\kappa}{\rho}h.
\end{equation}
\end{lemma}
\begin{proof}
\eqref{bounds_for_Mbar} can be verified directly. We only prove \eqref{eq2_13} for the $-$ case and the arguments for the $+$ case are similar. For simplicity, we denote $\mathbf{ n}(\widetilde{X}_i)=(\tilde{n}_{ix}, \tilde{n}_{iy})$, $i=1,2$. Then by direct calculations, we have
\begin{equation*}
\begin{split}
\overline{M}^-(\widetilde{X}_1)-M^-(\widetilde{X}_2) =&
\left(\begin{array}{cc}
\tilde{n}_{1y}\bar{n}_{y}-\tilde{n}^2_{2y}+\rho(\tilde{n}_{1x}\bar{n}_{x}-\tilde{n}^2_{2x}) & (\rho-1)(\tilde{n}_{1y}\bar{n}_x-\tilde{n}_{2x}\tilde{n}_{2y}) \\
(\rho-1)(\tilde{n}_{1x}\bar{n}_y-\tilde{n}_{2x}\tilde{n}_{2y})  & \tilde{n}_{1x}\bar{n}_{x}-n^2_{2x}+\rho(\tilde{n}_{1y}\bar{n}_{y}-\tilde{n}^2_{2y})
\end{array}\right)\\
&+\frac{1-\bar{\mathbf{ n}}\cdot\mathbf{ n}(\widetilde{X}_1)}{\bar{\mathbf{ n}}\cdot\mathbf{ n}(\widetilde{X}_1)}
\left(\begin{array}{cc}
\tilde{n}_{1y}\bar{n}_{y}+\rho\tilde{n}_{1x}\bar{n}_{x} & (\rho-1)\tilde{n}_{1y}\bar{n}_x \\
(\rho-1)\tilde{n}_{1x}\bar{n}_y  & \tilde{n}_{1x}\bar{n}_{x}+\rho\tilde{n}_{1y}\bar{n}_{y}
\end{array}\right).
\end{split}
\end{equation*}
By the triangular inequality, \eqref{eq2_7a}, \eqref{eq2_8b}, \eqref{h_small_assump}, and $\rho\leqslant1$, we can verify that $\| \overline{M}^-(\widetilde{X}_1)-M^-(\widetilde{X}_2)\| \leqslant C\kappa h$ for a constant $C$ independent of interface location, $\beta^{\pm}$, and $\kappa$.
\end{proof}

The following lemmas provide a group of identities on interface elements.
\begin{lemma}
\label{lem2_4}
For the mesh $\mathcal{T}_h$ with $h$ sufficiently small, the following results hold for all $\widetilde{X} \in \Gamma$:
\begin{itemize}
\item
$\overline{M}^-(\widetilde{X})$ and $\overline{M}^+(\widetilde{X})$ are inverse matrices to each other, i.e.,
\begin{equation}
\label{eq2_10}
\overline{M}^-(\widetilde{X})\overline{M}^+(\widetilde{X})=I,~~\overline{M}^+(\widetilde{X})\overline{M}^-(\widetilde{X}) = I.
\end{equation}

\item
Matrix $(\overline{M}^-(\widetilde{X}))^T$ has two eigenvalues $1$ and $\rho$ with the corresponding eigenvectors $\bar{\mathbf{ t}}$ and $\mathbf{ n}(\widetilde{X})$, i.e.,
\begin{equation}
\label{eq2_12}
\left( \overline{M}^-(\widetilde{X}) \right)^T \bar{\mathbf{ t}}=\bar{\mathbf{ t}},~~\left( \overline{M}^- (\widetilde{X})\right)^T \mathbf{ n}(\widetilde{X})=\rho\mathbf{ n}(\widetilde{X}).
\end{equation}

\item
Similarly, matrix $(\overline{M}^+(\widetilde{X}))^T$ has two eigenvalues $1$ and $1/\rho$ with the corresponding eigenvectors $\bar{\mathbf{ t}}$ and $\mathbf{ n}(\widetilde{X})$, respectively, i.e.,
\begin{equation}
\label{eq2_11}
\left( \overline{M}^+(\widetilde{X}) \right)^T \bar{\mathbf{ t}}=\bar{\mathbf{ t}},~~\left( \overline{M}^+(\widetilde{X}) \right)^T \mathbf{ n}(\widetilde{X})=\frac{1}{\rho}\mathbf{ n}(\widetilde{X}).
\end{equation}

\end{itemize}
\end{lemma}
\begin{proof} First it is easy to see that $\overline{M}^-(\widetilde{X})\overline{M}^+(\widetilde{X})=(\overline{N}^-(\widetilde{X}))^{-1}\overline{N}^+(\widetilde{X})(\overline{N}^+(\widetilde{X}))^{-1}\overline{N}^-(\widetilde{X})=I$. Next by direct calculation, we have
\begin{equation*}
\begin{split}
\overline{M}^-(\widetilde{X})
&=\frac{1}{\bar{n}_x\tilde{n}_x(\widetilde{X})+\bar{n}_y\tilde{n}_y(\widetilde{X})}
\left(\begin{array}{cc}
\bar{n}_y\tilde{n}_y(\widetilde{X})+\rho\bar{n}_x\tilde{n}_x(\widetilde{X}) & -\bar{n}_x\tilde{n}_y(\widetilde{X})+\rho\bar{n}_x\tilde{n}_y(\widetilde{X}) \\
-\bar{n}_y\tilde{n}_x(\widetilde{X})+\rho\bar{n}_y\tilde{n}_x(\widetilde{X}) & \bar{n}_x\tilde{n}_x(\widetilde{X})+\rho\bar{n}_y\tilde{n}_y(\widetilde{X})
\end{array}\right)
\end{split}
\end{equation*}
from which we can easily verify that $(\overline{M}^-(\widetilde{X}))^{T}\bar{\mathbf{ t}}=\bar{\mathbf{ t}}$ and
$ (\overline{M}^-(\widetilde{X}))^{T}\mathbf{ n}(\widetilde{X})=\rho \; \mathbf{ n}(\widetilde{X})$.
The results about $(\overline{M}^+(\widetilde{X}))^{T}$ follow from the fact $(\overline{M}^-(\widetilde{X}))^{T}(\overline{M}^+(\widetilde{X}))^{T}=I$.
\end{proof}

\begin{lemma}
\label{lem2_5}
Let $\mathcal{T}_h$ be a mesh with $h$ sufficiently small. Let $P\in \Omega$ and $\widetilde{X}$ be an arbitrary point on $\Gamma \cap T$. Then the following vectors are independent of $\overline{X}\in l$:
\begin{equation*}\label{}
(\overline{M}^+(\widetilde{X}) - I)^T(P - \overline{X}) \;\;\; and \;\;\; (\overline{M}^-(\widetilde{X}) - I)^T(P - \overline{X}).
\end{equation*}
\end{lemma}
\begin{proof}
For two arbitrary points $\overline{X}_i \in l, i = 1, 2$, $\overline{X}_1 - \overline{X}_2$ is a scalar multiple of $\bar{\mathbf{ t}}$. Hence, by Lemma \ref{lem2_4},
\begin{eqnarray*}
(\overline{M}^-(\widetilde{X}) - I)^T(P - \overline{X}_1) - (\overline{M}^-(\widetilde{X}) - I)^T(P - \overline{X}_2) = (\overline{M}^-(\widetilde{X}) - I)^T(\overline{X}_1 - \overline{X}_2) = 0
\end{eqnarray*}
which leads to $(\overline{M}^-(\widetilde{X}) - I)^T(P - \overline{X}_1) = (\overline{M}^-(\widetilde{X}) - I)^T(P - \overline{X}_2)$. Therefore
$(\overline{M}^-(\widetilde{X}) - I)^T(P - \overline{X})$ does not change when $\overline{X}\in l$ varies. The result for  $(\overline{M}^+(\widetilde{X}) - I)^T(P - \overline{X})$ can be proven similarly.
\end{proof}

\section{Multipoint Taylor Expansions on Interface Elements}
In this section, extending those in \cite{2009HeTHESIS,2008HeLinLin,2004LiLinLinRogers,1982Xu,2013ZhangTHESIS}, we derive multipoint Taylor expansions in more general formats for a function $u\in PC^2_{int}(T)$ over an arbitrary interface element $T\in\mathcal{T}^i_h$, in which $u(M_i), i \in \mathcal{I}$ is described in terms of $u$ and its derivatives at $x \in T^s, s = \pm$.  We also estimate the remainders in these expansions.
And as in \cite{2008HeLinLin}, we call a point $X\in T$ an obscure point if one of the lines $\overline{XM_i}, 1 \leq i \leq DOF(T)$ can intersect $\Gamma$ more than once. To facilitate a clear expository presentation of main ideas in our analysis, we carry out error estimation only for interface elements without any obscure points.
For the case containing obscure points, we can use a first order expansion for $u$ and use the argument that the measure of obscure points is bounded by $\mathcal{O}(h^3)$.


First, we partition $\mathcal{I}$ into two index sets: $\mathcal{I}^+=\{ i: M_i\in T^+ \}$ and $\mathcal{I}^-=\{ i: M_i\in T^- \}$ according to the locations of
$M_i, i \in \mathcal{I}$. For every $X\in T$, we let $Y_i(t,X)=tM_i+(1-t)X$. When
$X$ and $M_i$ are on different sides of $\Gamma$, we let $\tilde{t}_i=\tilde{t}_i(X)\in [0,1]$ such that $\widetilde{Y}_i=Y_i(\tilde{t}_i,X)$ is on the curve $\Gamma\cap T$, see Figure \ref{fig3_1} for an illustration in which the rotated-$Q_1$ finite elements are considered. When $M_i$ and $X$ are on the same side of $\Gamma$, by the standard second order Taylor expansion of $u\in PC^2_{int}(T)$, we have
\begin{eqnarray}
u^s(M_i)&=&u^s(X)+ \nabla u^s(X)\cdot(M_i-X) + R_i^s(X),~ i\in \mathcal{I}^s, s = \pm, ~\forall X \in T^s, \label{eq3_1}\\
\text{with~~} R_i^s(X) &=&\int_0^1(1-t)\frac{d^2}{dt^2}u^s(Y_i(t,X))dt, ~i\in \mathcal{I}^s, ~\forall X \in T^s. \label{eq3_2}
\end{eqnarray}

\begin{figure}[H]
\centering
\includegraphics[width=2.3in]{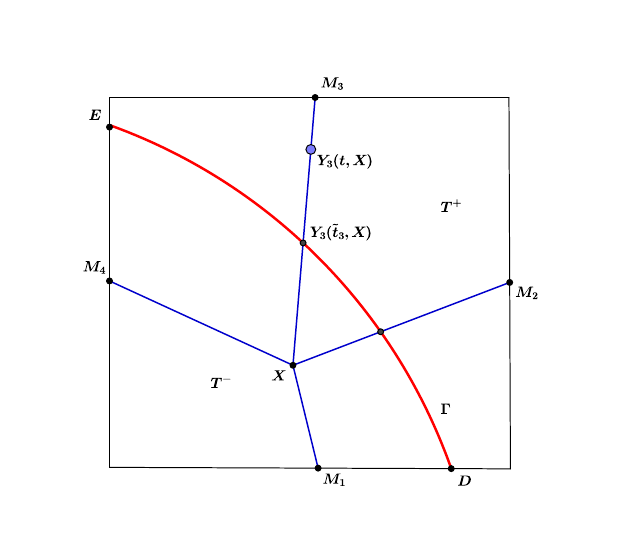}
\caption{ The expansion of $u$ in a rectangular interface element}
\label{fig3_1}
\end{figure}

In the following discussion, we denote $s = \pm, s' = \mp$, i.e., $s$ and $s'$ take opposite signs whenever a formula have them both.
When, $M_i$ and $X$ are on different sides of $\Gamma$, the expansions in \cite{2009HeTHESIS,2008HeLinLin,2004LiLinLinRogers,2013ZhangTHESIS} can be
generalized to the following format for $u\in PC^2_{int}(T)$:
\begin{equation}
\begin{split}
\label{eq3_3}
u^{s'}(M_i)=&u^{s}(X)+\nabla u^{s}(X)\cdot (M_i-X)+ \left(\left( M^{s}(\widetilde{Y}_i)-I \right)\nabla u^{s}(X)\right) \cdot (M_i-\widetilde{Y}_i)\\
& + R_i^{s}(X), ~i\in \mathcal{I}^{s'}, ~\forall X \in T^{s},
\end{split}
\end{equation}
with
\begin{equation}
\begin{cases} \label{eq3_4}
&R_i^s(X)=R_{i1}^s(X)+R_{i2}^s(X)+R_{i3}^s(X), ~i \in \mathcal{I}^{s'}, X \in T^s, \\
&R_{i1}^s(X)=\int_{0}^{\tilde{t}_i}(1-t)\frac{d^2u^s}{dt^2}(Y_i(t,X))dt, \;\;\; R_{i2}^s(X)=\int_{\tilde{t}_i}^{1}(1-t)\frac{d^2u^{s'}}{dt^2}(Y_i(t,X))dt, ~i \in \mathcal{I}^{s'}, X \in T^s,\\
&R_{i3}^s(X)=(1-\tilde{t}_i)\int_{0}^{\tilde{t}_i}\frac{d}{dt} \left( (M^s(\widetilde{Y}_i)-I)\nabla u^s(Y_i(t,X))\cdot (M_i-X) \right)dt, ~i \in \mathcal{I}^{s'}, X \in T^s,
\end{cases}
\end{equation}
where $M^s(\widetilde{Y}_i)$ are from \eqref{eq2_3} and \eqref{eq2_4}. We proceed to estimate remainders in \eqref{eq3_1} and \eqref{eq3_3}.

\begin{lemma}
\label{lem3_2}
Assume $u\in PC^2_{int}(T)$. Then there exist constants $C>0$ independent of the interface location and $\beta^{\pm}$ such that
\begin{equation}
\label{eq3_5}
   \int_{T^s}(1-t)^2|u_{d_1d_2}(Y_i(t,X))|^2dX \leqslant C|u|^2_{2,T},~s = \pm, ~\forall t \in [0, 1],
\end{equation}
where $d_1, d_2 =x$ or $y$.
\end{lemma}
\begin{proof}
Let $M_i = (x_i, y_i), i \in \mathcal{I}$ and let $\xi=t x_i+(1-t)x$ and $\eta=t y_i+(1-t)y$.
For each fixed $t \in [0,1]$, define
\begin{equation}
T^s(t)=\{ tM_i+(1-t)X ~|~ X\in T^s \}.
\end{equation}
Since $T^s(t) \subseteq T, s = \pm$, we have
\begin{eqnarray*}
 \int_{T^s}(1-t)^2|u_{d_1d_2}(Y(t,X))|^2dX
= \int_{T^s(t)}(1-t)^2 u^2_{d_1d_2}(\xi,\eta)(1-t)^{-2}d\xi d\eta \leqslant |u|^2_{2,T}
\end{eqnarray*}
which leads to \eqref{eq3_5}. \end{proof}
%

By a direct calculation we have
\begin{equation}
\begin{split}
\label{eq3_7}
\frac{d^2}{dt^2}u^s(Y_i(t,X))=(M_i-X)^TH_u^s(Y_i(t,X))(M_i-X), ~s = \pm,
\end{split}
\end{equation}
where
\begin{equation*}
H_u^s(Y_i(t,X))=
\left(\begin{array}{cc}
u_{xx}^s(Y_i(t,x)) & u_{xy}^s(Y_i(t,x)) \\
u_{yx}^s(Y_i(t,x)) & u_{yy}^s(Y_i(t,x))
\end{array}\right), ~s = \pm,
\end{equation*}
is the Hessian matrix of $u^s$. We are now ready to derive bounds for the remainders in the following lemmas.

\begin{lemma}
\label{lem3_3}
Assume $u\in PC^2_{int}(T)$, there exist constants $C>0$ independent of the location of the interface and $\beta^{\pm}$ such that
\begin{eqnarray}
&&\| R_i^s \|_{0,T^s} \leqslant Ch^2 |u|_{2,T}, \;\;\; i\in \mathcal{I}^s, ~s = \pm, \label{eq3_8} \\
&&\| R_{i1}^s \|_{0,T^s} \leqslant Ch^2 |u|_{2,T}, \;\;\; \| R_{i2}^s \|_{0,T^s}  \leqslant Ch^2 |u|_{2,T}, \;\;\; i \in \mathcal{I}^{s'}, ~s = \pm. \label{eq3_9}
\end{eqnarray}
\end{lemma}
\begin{proof}
According to \eqref{eq3_2} and \eqref{eq3_7}, for $s = \pm$, we have
\begin{equation*}
\begin{split}
\| R_i^s \|_{0,T^s}=& \left(  \int_{T^s} \left( \int_{0}^{1}(1-t)(M_i-X)H_u^s(Y_i(t,X))(M_i-X)dt \right)^2 dX  \right)^{\frac{1}{2}}\\
\leqslant& Ch^2  \int_{0}^{1} \left( \int_{T^s}(1-t)^2(|u_{xx}^s(Y_i(t,X))|^2+|u_{xy}^s(Y_i(t,X))|^2+|u_{yy}^s(Y_i(t,X))|^2)dX\right)^{\frac{1}{2}}dt.
\end{split}
\end{equation*}
Then \eqref{eq3_8} follows from Lemma \ref{lem3_2}. Estimate \eqref{eq3_9} can be derived similarly.
\end{proof}


Moreover, note that for $i\in\mathcal{I}^{s'}, X \in T^s$ and $t\in [0, \tilde{t}_i(X)]$, it can be verified that
\begin{equation}
\begin{split}
\label{eq3_10}
\frac{d}{dt}\left( (M^{s}(\widetilde{Y}_i)-I)\nabla u^s(Y_i(t,X))\cdot(M_i-X) \right)= (M_i-X)^T H_u^s(Y_i(t,X))(M^{s}(\widetilde{Y}_i)-I)^T(M_i-X).
\end{split}
\end{equation}

\begin{lemma}
\label{lem3_5}
Assume $u\in PC^2_{int}(T)$, there exist constants $C>0$ independent of the interface location and $\beta^{\pm}$ such that
\begin{equation}
\label{eq3_11}
\| R_{i3}^- \|_{0,T^-}\leqslant Ch^2 |u|_{2,T}, \;\;\; i\in \mathcal{I}^{+}, ~~~~ \| R_{i3}^+ \|_{0,T^+}\leqslant \frac{C}{\rho}h^2 |u|_{2,T}, \;\;\; i\in \mathcal{I}^{-}.
\end{equation}
\end{lemma}
\begin{proof}
We only provide the proof of $\| R_{i3}^- \|_{0,T^-}$ and the argument for $\| R_{i3}^+ \|_{0,T^+}$ is similar. According to \eqref{eq3_10} and the fact $0\leqslant1-\tilde{t}_i(X)\leqslant 1-t$ for any $t\in [0,\tilde{t}_i(X)]$, we have
\begin{equation*}
\begin{split}
\|R^-_{i3}\|_{0,T^-}
=&\left( \int_{T^-} (1-\tilde{t}_i(X))^2 \left( \int_{0}^{\tilde{t}_i(X)}  (M_i-X)^T H_u^-(Y_i(t,X))(M^{-}(\widetilde{Y}_i)-I)^T(M_i-X)dt \right)^2 dX \right)^{\frac{1}{2}}\\
\leqslant& Ch^2 \left( \int_{T^-} \left( \int_0^1|1-t| \; \big(\abs{u^-_{xx}(Y_i(t,X))} + \abs{u^-_{xy}(Y_i(t,X))} + \abs{u^-_{yy}(Y_i(t,X))}\big)dt\right)^2dX \right)^{\frac{1}{2}}\\
\leqslant& Ch^2 \int_{0}^{1} \left( \int_{T^-}(1-t)^2\big(|u^-_{xx}(Y_i(t,X))|^2+|u^-_{xy}(Y_i(t,X))|^2+|u^-_{yy}(Y_i(t,X))|^2\big)dX\right)^{\frac{1}{2}}dt,
\end{split}
\end{equation*}
where we use the fact $\rho\leqslant1$. Then applying the estimates in Lemma \ref{lem3_2} to the above leads to \eqref{eq3_11}.
\end{proof}

\section{IFE Spaces and Their Properties}
\label{sec:IFE Spaces and Their Properties}

In this section, we discuss IFE spaces constructed from the related finite elements $(T,\Pi_T,\Sigma^P_T)$ for $T \in \mathcal{T}_h$
described in \eqref{eq4_1} and \eqref{eq4_2}. We will first address the unisolvence of the immersed finite elements on interface elements. We will then present a few fundamental properties of IFE functions. Moreover, we will show that these IFE spaces have the optimal approximation capability according to the polynomials used to constructed them.

\subsection{Local IFE spaces}


First, on each element $T \in \mathcal{T}_h$, the standard finite element $(T,\Pi_T,\Sigma^P_T)$ leads to the following local finite element space:
\begin{equation}
\label{eq4_3}
S^{P}_h(T)= \textrm{Span} \{ \psi^P_{i,T}:\; i\in\mathcal{I} \},
\end{equation}
where $\psi^P_{i,T}, i\in\mathcal{I}$ are the shape functions satisfying \eqref{eqn_psi_iT_values_bounds}. This local finite element space is then naturally used as the local
IFE space on every non-interface element $T \in \mathcal{T}_h^n$. Therefore, our effort here focuses on the local IFE space on interface elements. We will discuss the unisolvence, i.e., we will show that the local degrees of freedom $\Sigma^P_T$ can uniquely determine an IFE function with a suitable set of interface jump conditions. The unisolvence guarantees the existence and uniqueness of IFE shape functions that can span the local IFE space on interface elements.

Let $T\in \mathcal{T}^i_h$ be a typical interface element with vertices $A_i, i \in \mathcal{I}$. Without loss of generality, we assume
\begin{eqnarray}\label{eq4_4}
\begin{array}{ll}
A_1=(0,0),~~A_2=(h,0),~~A_3=(h,h), & \text{($T$ is a triangular element)}, \\
A_1=(0,0),~~A_2=(h,0),~~A_3=(h,h),~~A_4=(0,h), & \text{($T$ is a rectangular element)},
\end{array}
\end{eqnarray}
and the edges of $T$ are denoted as
\begin{eqnarray}\label{eq4_5}
\begin{array}{ll}
b_1=\overline{A_1A_2},~~b_2=\overline{A_2A_3},~~b_3=\overline{A_3A_1}, & \text{($T$ is a triangular element)}, \\
b_1=\overline{A_1A_2},~~b_2=\overline{A_2A_3},~~b_3=\overline{A_3A_4},~~b_4=\overline{A_4A_1}, & \text{($T$ is a rectangular element)}.
\end{array}
\end{eqnarray}


On each interface element $T$, we consider IFE functions in the following piecewise polynomial format:
\begin{equation}
\label{eq4_10}
\phi^P_T(X) =
\left\{
\begin{aligned}
\phi^{P\;-}_T(X)=\phi^{-}_T(X)\in \Pi_T \;\;\;\;\; & \text{if} \;\; X\in T^-, \\
\phi^{P\;+}_T(X)=\phi^{+}_T(X)\in \Pi_T \;\;\;\;\; & \text{if} \;\; X\in T^+,
\end{aligned}
\right.
\end{equation}
such that it can satisfy the jump conditions \eqref{eq1_3} and \eqref{eq1_4} in an approximate sense as follow:
\begin{eqnarray}
&&\begin{cases}
\phi^{-}_T|_{l}=\phi^{+}_T|_{l}, &\text{($T$ is a triangular element)}, \\
\phi^{-}_T|_{l}=\phi^{+}_T|_{l}, ~~d(\phi^{-}_T)=d(\phi^{+}_T), &\text{($T$ is a rectangular element)},
\end{cases} \label{eq4_8} \\
&&\beta^-\nabla\phi^{-}_T(F)\cdot\mathbf{ n}(F)=\beta^+\nabla\phi^{+}_T(F)\cdot\mathbf{ n}(F), \label{eq4_9}
\end{eqnarray}
where $d(p)$ denotes the coefficient in the second degree term for $p \in \Pi_T$ and $F$ is an arbitrary point on $\Gamma \cap T$.
For an IFE function $\phi^P_T$ such that
\begin{equation}
\label{eq4_12}
\phi^P_T(M_i)=v_i, \;\;\;\; i\in \mathcal{I},
\end{equation}
we can first expand $\phi^P_T$ on the sub-element with more degrees of freedom, i.e., on $T^+$ with the assumption that
$\abs{\mathcal{I}^+} \geq \abs{\mathcal{I}^-}$ without loss of generality, and the condition \eqref{eq4_8} then implies that
\begin{equation}
\label{eq4_13}
\phi^P_T(X) =
\begin{cases}
 \phi^{P\;-}_T(X)  = \phi^{P\;+}_T(X)+c_0L(X) & \text{if} \;\; X\in T^-, \\
 \phi^{P\;+}_T(X)  = \sum_{i\in\mathcal{I}^-}c_i\psi^P_{i,T}(X)+\sum_{i\in\mathcal{I}^+}v_i\psi^P_{i,T}(X)& \text{if} \;\; X\in T^+,
\end{cases}
\end{equation}
where the function
\begin{equation}
\label{eq4_11}
L(X)=\bar{\mathbf{ n}}\cdot(X-D)
\end{equation}
is such that $L(X)=0$ is the equation of the line $l$ and $\nabla L(X) = \bar{\mathbf{n}}$.

Recall from Remark \ref{rem2_1},
$
\nabla L(F)\cdot\mathbf{ n}(F)=\bar{\mathbf{n}}\cdot\mathbf{ n}(F)\geqslant 1-\bar{\kappa} >0,
$
when $h$ is small enough; hence
\begin{eqnarray}
\label{k}
\mu = \left( \frac{1}{\rho}-1 \right)\frac{ 1}{\bar{\mathbf{n}}\cdot\mathbf{ n}(F)}
\end{eqnarray}
is well defined, and, by $\rho\leqslant1$, we have
\begin{eqnarray}
\label{bound_for_k}
0\leqslant \mu  \leqslant \left( \frac{1}{\rho}-1 \right) \frac{1}{1 - \bar{\kappa}}.
\end{eqnarray}
By condition \eqref{eq4_9}, we then have
\begin{equation}
\label{eqn4_18}
c_0=
\mu \left( \sum_{i\in\mathcal{I}^-}c_i\nabla\psi^P_{i,T}(F)\cdot\mathbf{ n}(F)+\sum_{i\in\mathcal{I}^+}v_i\nabla\psi^P_{i,T}(F)\cdot\mathbf{ n}(F) \right).
\end{equation}
Putting this formula for $c_0$ in formula \eqref{eq4_13} for $\phi_T^P(X)$ and setting
$\phi^{P\;-}_T(M_j)=v_j$ for $j \in \mathcal{I}^-$ leads to the following linear system for $c_i, i \in \mathcal{I}^-$:
\begin{equation}
\begin{split}
\label{eqn4_20}
&\sum_{i\in\mathcal{I}^-} \left( \psi^P_{i,T}(M_j) + \mu \nabla\psi^P_{i,T}(F)\cdot\mathbf{ n}(F) L(M_j) \right)c_i \\
=& v_j-\sum_{i\in\mathcal{I}^+} \left( \psi^P_{i,T}(M_j) + \mu \nabla\psi^P_{i,T}(F)\cdot\mathbf{ n}(F) L(M_j) \right)v_i,~~j \in \mathcal{I}^-.
\end{split}
\end{equation}
Since $\psi^P_{i,T}(M_j)=\delta_{ij}$, for $i,j\in\mathcal{I}^-$, we can write the linear system \eqref{eqn4_20} in the following matrix form:
\begin{equation}
\label{eqn4_15}
(I + \mu \,\bfdelta \bfgamma^T )\bfc = \bfb,
\end{equation}
where $\mathbf{c}=(c_i)_{i\in \mathcal{I}^-}$,
\begin{equation}
\label{eqn4_16}
\bfgamma = \left( \nabla\psi^P_{i,T}(F)\cdot\mathbf{ n}(F) \right)_{i\in\mathcal{I}^-}, \;\;\; \bfdelta=\left( L(M_i) \right)_{i\in\mathcal{I}^-}
\end{equation}
and
\begin{equation}
\label{eqn4_17}
\bfb =
\left(
v_i- \mu L(M_i) \sum_{j\in\mathcal{I}^+} \nabla\psi^P_{j,T}(F)\cdot\mathbf{ n}(F)v_j
 \right)_{i\in\mathcal{I}^-}
\end{equation}
are all column vectors. We proceed to show that $\phi_T^P(X)$, i.e., its coefficients $c_0, \bfc$ are uniquely determined. We need the following two lemmas. Let
$\bar{\bfgamma}=\left( \nabla\psi_{i,T}(F_{\bot})\cdot \bar{\bfn} \right)_{i\in\mathcal{I}^-}$.
\begin{lemma}
\label{lemma4_3}
For all the interface elements, we have
$
\bar{\bfgamma}^T\bfdelta \in [0,1]
$.
And for the linear and bilinear $\psi_{i, T}, i\in\mathcal{I}$, there holds
\begin{equation}
\label{eq_Ldelgam_1}
|L(M_i)|<2h \, \sqrt{\bar{\bfgamma}^T\bfdelta}, ~~~ \forall i\in\mathcal{I}^-.
\end{equation}
Furthermore, for the bilinear $\psi_{i, T}, i\in\mathcal{I}$, if $F$ is chosen to be such that $F_{\bot}$ is the mid-point of the line $\overline{DE}$, then
\begin{equation}
\label{eq_Ldelgam_2}
|L(M_i)|<2h \, \bar{\bfgamma}^T\bfdelta, ~~~ \forall i\in\mathcal{I}^-.
\end{equation}
\end{lemma}
\begin{proof}
We only give the proof for the case in which $\Pi_T$ is the rotated-$Q_1$ polynomial space, the interface element $T$ is such that $\mathcal{I}^-=\{ 1 \}$ and $\mathcal{I}^+=\{2,3,4\}$ with $D=(hd, 0)$ and $E=(0,he)$ for some $d\in[1/2,1]$ and $e\in[0,1/2]$. Similar arguments apply to all other cases. First, $\bar{\mathbf{ n}}=(e,d)/\sqrt{d^2+e^2}$. Hence, $F_{\bot}=(td,e(h-t))$ for some $t\in (0,h)$. By direct calculation, we have
\begin{equation}
\begin{split}
\label{}
 \bar{\bfgamma}^T \bfdelta= \frac{1}{h(d^2+e^2)}\left[ h(e-2d)-2(2t-h)de \right] (\frac{1}{2}-d)e,
\end{split} \label{eq_lemma5.1_1}
\end{equation}
Note that
\begin{equation*}
\label{}
\bar{\bfgamma}^T \bfdelta = \frac{e-2d+2de}{d^2+e^2}(\frac{1}{2}-d)e\in [0,1] \;\;\text{if} \;\; t=0 \;\;\; \text{and} \;\;\;
\bar{\bfgamma}^T \bfdelta = \frac{e-2d-2de}{d^2+e^2}(\frac{1}{2}-d)e\in [0,1] \;\;\text{if} \;\; t=h,
\end{equation*}
which leads to $\bar{\bfgamma}^T\bfdelta \in [0,1]$ because $\bar{\bfgamma}^T \bfdelta$ is a linear function of $t$ according to \eqref{eq_lemma5.1_1}. And \eqref{eq_Ldelgam_1} and \eqref{eq_Ldelgam_2} follow from similar calculation.
\end{proof}


\begin{lemma}
\label{lemma4_4}
For $h$ small enough, we have
\begin{equation}
\label{eqn4_22}
1+\mu \,\bfgamma^T\bfdelta 
  \geqslant 1- 4\mu\sqrt{\bar{\kappa}}.
\end{equation}
And for the linear and Crouzeix-Raviart $\psi_{i, T}, i\in\mathcal{I}$, if $F$ is chosen such that $\mathbf{n}(F)=\bar{\mathbf{ n}}$, then
\begin{equation}
\label{eq_lower_bound_2}
1+\mu\bfgamma^T\bfdelta=1+\mu\bar{\bfgamma}^T\bfdelta\geqslant1.
\end{equation}
In addition, for the bilinear $\psi_{i, T}, i\in\mathcal{I}$, if $F$ is chosen such that $F_{\bot}$ is the midpoint of $\overline{DE}$, then
\begin{equation}
\label{eq_lower_bound_3}
1+\mu\bfgamma^T\bfdelta\geqslant1+\mu\bar{\bfgamma}^T\bfdelta(1-4\sqrt{2}\sqrt{\bar{\kappa}}).
\end{equation}
\end{lemma}
\begin{proof}
By Lemma \ref{lem2_3}, \eqref{eqn_psi_iT_values_bounds}, and Remark \ref{rem2_1}, we have
\begin{equation*}
\begin{split}
\label{}
& |\nabla\psi^P_{i,T}(F)\cdot\mathbf{ n}(F) -  \nabla\psi^P_{i,T}(F_{\bot})\cdot \bar{\bfn}| \\
\leqslant& |(\nabla\psi^P_{i,T}(F) -  \nabla\psi^P_{i,T}(F_{\bot}))\cdot\mathbf{ n}(F)|+|\nabla\psi^P_{i,T}(F_{\bot})\cdot(\mathbf{ n}(F)- \bar{\bfn})| \\
 \leqslant& \|(\nabla\psi^P_{i,T}(F) -  \nabla\psi^P_{i,T}(F_{\bot}))\|\cdot \|\mathbf{ n}(F)\|+ \|\nabla\psi^P_{i,T}(F_{\bot})\|\cdot \|(\mathbf{ n}(F)- \bar{\bfn})\| \\
\leqslant&  4(1+(1-2\epsilon^2)^{-3/2})\kappa, ~i \in \mathcal{I}
\end{split}
\end{equation*}
which implies $\norm{\bfgamma - \bar{\bfgamma}} \leq 2(1+2(1-2\epsilon^2)^{-3/2})\kappa$. For all the types of finite elements considered in this article, by $\|L\|_{\infty,T} \leq \sqrt{2}h$, we have $\norm{\bfdelta} \leq \sqrt{2}h$ and therefore,
\begin{equation*}
\begin{split}
\label{}
 \mu \bfgamma^T\bfdelta = \mu \bar{\bfgamma}^T\bfdelta + \mu(\bfgamma^T - \bar{\bfgamma}^T)\bfdelta \geq \mu \bar{\bfgamma}^T\bfdelta- 4\sqrt{2}(1+(1-2\epsilon^2)^{-3/2})\kappa \mu h,
\end{split}
\end{equation*}
which yields \eqref{eqn4_22} by the assumption \eqref{h_small_assump}. Furthermore, for linear finite elements, if $F$ is chosen such that $\mathbf{ n}(F)=\bar{\mathbf{ n}}$, then $\bar{\bfgamma}=\bfgamma$; thus Lemma \ref{lemma4_3} and \eqref{bound_for_k} imply $1+\mu \,\bar{\bfgamma}^T\bfdelta\geqslant1$. For the bilinear finite elements, if $F$ is chosen such that $F_{\bot}$ is  the mid point of $\overline{DE}$, then by \eqref{eq_Ldelgam_2} we have
\begin{eqnarray*}
 \mu \bfgamma^T\bfdelta = \mu \bar{\bfgamma}^T\bfdelta + \mu(\bfgamma^T - \bar{\bfgamma}^T)\bfdelta
&\geq& \mu \bar{\bfgamma}^T\bfdelta - \mu 4(1+(1-2\epsilon^2)^{-3/2})\kappa (2h\bar{\bfgamma}^T\bfdelta) \\
&\geq& \mu\bar{\bfgamma}^T\bfdelta ( 1 - 8(1+(1-2\epsilon^2)^{-3/2})\kappa h ),
\end{eqnarray*}
which leads to \eqref{eq_lower_bound_3} by the assumption \eqref{h_small_assump}.
\end{proof}

\begin{thm}[$\mathbf{ Unisolvence}$]
\label{thm4_1}
Let $\mathcal{T}_h$ be a mesh satisfying \eqref{h_small_assump} with $\bar{\kappa}$ specified therein for linear and Crouzeix-Raviart $\psi_{i, T}, i\in\mathcal{I}$ and
\begin{equation}
\sqrt{\bar{\kappa}} < \frac{1}{4\sqrt{2}}, \text{~~for bilinear~~} \psi_{i, T}, i\in\mathcal{I}, \label{rem_unisol_eq2_1}
\end{equation}
and for some $\lambda \in (0, 1)$
\begin{equation}
\sqrt{\bar{\kappa}} \leq \frac{\rho(1- \lambda)}{4 - (3 + \lambda)\rho}, \text{~~for rotated-$Q_1$~~} \psi_{i, T}, i\in\mathcal{I}. \label{rem_unisol_eq1_1}
\end{equation}
In addition, we assume $F$ and $F_{\bot}$ are chosen such that estimates given by Lemma \ref{lemma4_4} hold. Then, given any vector $v=(v_1,v_2,v_3,v_4)\in \mathbb{R}^4$ for the bilinear and rotated $Q_1$ case (or $v=(v_1,v_2,v_3)\in \mathbb{R}^3$ for the linear and Crouzeix-Raviart case), there exists one and only one IFE function $\phi^P_T$ in the form of \eqref{eq4_10} satisfying \eqref{eq4_8}-\eqref{eq4_12}.
\end{thm}
\begin{proof}
For linear and Crouzeix-Raviart $\psi_{i, T}, i\in\mathcal{I}$, by \eqref{eq_lower_bound_2} in Lemma \ref{lemma4_4}, we have $1+\mu \bfgamma^T\bfdelta \neq 0$.
For bilinear $\psi_{i, T}, i\in\mathcal{I}$, by \eqref{rem_unisol_eq2_1}, we have
$
1 - 4\sqrt{2}\sqrt{\bar{\kappa}} > 0
$
which leads to $1+\mu \bfgamma^T\bfdelta \neq 0$ because of \eqref{eq_lower_bound_3} in Lemma \ref{lemma4_4}. Similarly,
for rotated-$Q_1$ $\psi_{i, T}, i\in\mathcal{I}$, by \eqref{rem_unisol_eq1_1}, we have
\begin{align}
1 - 4\mu \sqrt{\bar{\kappa}} \geq 1-\frac{4\sqrt{\bar{\kappa}}}{1-\bar{\kappa}}\left(\frac{1}{\rho}-1\right) \geq \lambda \label{rem_unisol_eq3}
\end{align}
which, by \eqref{eqn4_22} in Lemma \ref{lemma4_4}, leads to $1+\mu \bfgamma^T\bfdelta \neq 0$ again.
Hence, by the well known results about the \textit{Sherman-Morrison} formula, the matrix in the linear system \eqref{eqn4_15} is nonsingular which together with \eqref{eqn4_18} lead to the existence and uniqueness for coefficients $c_i, i \in \mathcal{I}^-$ and $c_0$ of $\phi_T^P(X)$.
\end{proof}

\begin{rem}
\label{rem_unisol}
Theorem \ref{thm4_1} provides guidelines on the choice for the angle allowance parameter $\bar{\kappa}$ needed in \eqref{h_small_assump} for
bilinear and rotated-$Q_1$ $\psi_{i, T}, i\in\mathcal{I}$. In the bilinear case, condition \eqref{rem_unisol_eq2_1} suggests an upper bound for $\bar{\kappa}$
which is nevertheless independent of $\rho$. In the rotated-$Q_1$ case, condition \eqref{rem_unisol_eq1_1} leads to the following upper bound for $\bar{\kappa}$:
$\sqrt{\bar{\kappa}}<\frac{\rho}{4- 3\rho}$
which depends on $\rho$, and this restriction on $\bar{\kappa}$ becomes severer when $\rho$ approaches $0$.
\end{rem}

\begin{rem}
\label{rem4_1}
Under the conditions given in Theorem \ref{thm4_1}, we can apply the \textit{Sherman-Morrison} formula to express the solution to \eqref{eqn4_15} explicitly as
\begin{equation}
\label{eqn4_23}
\bfc = \bfb - \mu \frac{(\bfgamma^T \bfb)\bfdelta}{1 + \mu \bfgamma^T\bfdelta}
\end{equation}
which facilitates both analysis and computation for these IFE spaces.
\end{rem}


On each interface element $T$, Theorem \ref{thm4_1} guarantees the existence and uniqueness of the IFE shape functions $\phi^P_{i,T}$, $i\in \mathcal{I}$
satisfying \eqref{eq4_10}-\eqref{eq4_9} such that
\begin{equation}
\label{eq4_21}
\phi^P_{i,T}(M_j)=\delta_{ij},~~i, j \in \mathcal{I},
\end{equation}
where $\delta_{ij}$ is the Kronecker delta function. Therefore, we can define the local IFE space on each interface element $T\in\mathcal{T}^i_h$ as
\begin{equation}
\label{eq4_22}
S^{P}_h(T)=\textrm{Span}\{ \phi^P_{i,T}:\;i\in \mathcal{I} \}.
\end{equation}



\commentout{
As in \cite{2013ZhangTHESIS}, depending on whether the interface cuts two adjacent or opposite edges, there are two types of interface elements.
According the location of $M_i, i \in \mathcal{I}$, a Type I element can be in $3$ cases and a Type II element can be in $2$ cases, see Figure \ref{fig_noncon_P_TypeI} and Figure \ref{fig_noncon_P_TypeII} for
illustrations of typical configurations.

 Let $F$ be an arbitrary point on $\Gamma\cap T$. We emphasize that, once chosen, $F$ should be fixed for each interface element.
\begin{figure}[H]
\centering
\includegraphics[width=6in]{figure5_1-eps-converted-to.pdf}
\caption{ \textbf{Type I} interface element and 3 cases }
\label{fig_noncon_P_TypeI}
\end{figure}

\begin{figure}[H]
\centering
\includegraphics[width=4in]{figure5_2-eps-converted-to.pdf}
\caption{ \textbf{Type II} interface element and 2 cases }
\label{fig_noncon_P_TypeII}
\end{figure}
}



\subsection{Properties of the IFE Shape Functions}

In this section, we present some fundamental properties of IFE shape functions.

\begin{thm}[Bounds of IFE shape functions]
\label{thm4_2}
Under the conditions given in Theorem \ref{thm4_1}, we have the following estimates:
\begin{itemize}
\item
For rotated-$Q_1$ and Crouzeix-Raviart $\phi^P_{i,T}$, $i\in \mathcal{I},~\forall \,T \in \mathcal{T}^i_h$,
\begin{equation}
\label{eq4_24}
|\phi^P_{i,T}|_{k,\infty,T}\leqslant  \frac{C}{\rho}h^{-k}, k= 0, 1, 
\end{equation}
where $C$ depends also on $\lambda$ for rotated-$Q_1$ case.
\item
For linear $\phi^P_{i,T}$, $i\in \mathcal{I}, ~\forall \,T \in \mathcal{T}^i_h$,
\begin{equation}
\label{linear_basis_bounds}
|\phi^P_{i,T}|_{k,\infty,T^+}\leqslant  \frac{C}{\sqrt{\rho}}h^{-k}, ~~|\phi^P_{i,T}|_{k,\infty,T^-}\leqslant  \frac{C}{\rho}h^{-k}, ~~ k=0,1.
\end{equation}
\item
For bilinear $\phi^P_{i,T}$, $i\in \mathcal{I}, ~\forall \,T \in \mathcal{T}^i_h$,
\begin{equation}
\label{bilinear_basis_bounds_2}
|\phi^P_{i,T}|_{k,\infty,T^+}\leqslant  Ch^{-k}, ~~~ |\phi^P_{i,T}|_{k,\infty,T^-}\leqslant  \frac{C}{\rho}h^{-k}, ~~ k=0,1.
\end{equation}
\end{itemize}
\end{thm}
\begin{proof}
For convenience, we let $\mathbf{ e}=(e_{j})_{j\in\mathcal{I}}$ be the unit vector constructing the basis functions $\phi^P_{i,T}$, which could be $(1,0\cdots,0)$, $(0,1\cdots,0)$, $\cdots$, $(0,0\cdots,1)$, and $\mathbf{ e}^-=(e_{j})_{j\in\mathcal{I}^-}$. Let $w=\sum_{j\in\mathcal{I}^+} \nabla\psi^P_{j,T}(F)\cdot\mathbf{ n}(F)e_j$. Then \eqref{eqn4_17} implies $\bfb=\mathbf{ e}^- - \mu w \bfdelta$ and plugging it into the \textit{Sherman}-\textit{Morrison} formula \eqref{eqn4_23} leads to
\begin{equation}
\label{bounded_1}
\begin{split}
\bfc 
 = \mathbf{e}^- - \mu \frac{(w + \bfgamma^T\mathbf{ e}^-)\bfdelta }{1 + \mu \bfgamma^T\bfdelta },
\end{split}
\end{equation}
and plugging \eqref{bounded_1} into \eqref{eqn4_18} yields
\begin{equation}
\label{bounded_2}
c_0=\frac{\mu (w + \bfgamma^T\mathbf{ e}^-)}{1 + \mu \bfgamma^T\bfdelta}.
\end{equation}
Since
$
\|\nabla\psi^P_{i,T}\|_{\infty,T}\leqslant Ch^{-1}, ~i\in \mathcal{I}, ~\|L\|_{\infty,T}<Ch, ~|\nabla L\|_{\infty,T}<C \label{eq_thm4_2_1}
$
for some constants $C$ independent of the location of the interface and $\beta^{\pm}$, we have $\norm{\bfgamma} \leq Ch^{-1}, ~\norm{\bfdelta} \leq Ch$, $\norm{\bfb} \leq C$ and $|w|\leqslant Ch^{-1}$. \\

When $\phi^P_{i,T}$, $i\in \mathcal{I}$ are rotated-$Q_1$ polynomials, we can apply \eqref{bound_for_k} and \eqref{rem_unisol_eq1_1} to \eqref{bounded_1} to obtain:
\begin{equation}
\begin{split}
\label{eqn3_13}
\|\bfc\| \leq \|\mathbf{ e}^-\|+ \frac{1}{1- \bar{\kappa}}\left( \frac{1}{\rho}-1 \right) \frac{ (|w|+\| \bfgamma \| \| \mathbf{ e}^-\|)\|\bfdelta\| }{1-\frac{4\sqrt{\bar{\kappa}}}{1-\bar{\kappa}}\left( \frac{1}{\rho}-1 \right)} \leq \frac{C}{\rho}.
\end{split}
\end{equation}
Applying similar arguments to \eqref{bounded_2} we have $\|c_0\|<C/\rho$. Constant $C$ in these inequalities depends on $\lambda$. Finally, \eqref{eq4_24} follows from applying these bounds for $\bfc$ and $\bfc_0$ and the bounds for standard finite element basis functions $\psi_{i,T}$ to the formula of $\phi^P_{i,T}$ given in \eqref{eq4_13}. When $\phi^P_{i,T}$, $i\in \mathcal{I}$ are
Crouzeix-Raviart polynomials, we apply \eqref{bound_for_k} and \eqref{eq_lower_bound_2} to \eqref{bounded_1} and \eqref{bounded_2} to obtain:
\begin{eqnarray}
\|\bfc\| &\leq&  \|\mathbf{ e}^-\|+ \frac{1}{1- \bar{\kappa}}\left( \frac{1}{\rho}-1 \right) (|w|+\| \bfgamma \| \| \mathbf{ e}^-\|)\|\bfdelta\|  \leq \frac{C}{\rho}, \nonumber \\
| c_0 | &\leq& \frac{1}{1- \bar{\kappa}}\left( \frac{1}{\rho}-1 \right) (|w|+\| \bfgamma \| \| \mathbf{ e}^-\|)\|\bfdelta\| \leq  \frac{C}{\rho}. \label{bounded_7}
\end{eqnarray}
Then \eqref{eq4_24} in this case follows from the same arguments used above for the rotated-$Q_1$ case.

When $\phi^P_{i,T}$, $i\in \mathcal{I}$ are linear polynomials, by using \eqref{eq_Ldelgam_1} and Lemma \ref{lemma4_4} in \eqref{bounded_1}, we have
\begin{equation}
\label{bounded_3}
\| \bfc \| \leq C+C\frac{\mu\sqrt{\bar{\bfgamma}^T\bfdelta}}{1+\mu\bar{\bfgamma}^T\bfdelta} \leq C+C\max{ (1, \sqrt{\mu}) } \leq \frac{C}{\sqrt{\rho}}.
\end{equation}
Also, estimate in \eqref{bounded_7} is valid. Then, estimates in \eqref{linear_basis_bounds} follow from applying \eqref{bounded_3} and \eqref{bounded_7} to \eqref{eq4_13}.
%
%

Finally, when $\phi^P_{i,T}$, $i\in \mathcal{I}$ are bilinear polynomials, we can apply \eqref{eq_Ldelgam_2} and Lemma \ref{lemma4_4} in \eqref{bounded_1} to have
\begin{equation}
\label{bounded_5}
\| \bfc \|<C+C\frac{\mu\bar{\bfgamma}^T\bfdelta}{1 + (1-4\sqrt{2}\sqrt{\bar{\kappa}}) \mu\bar{\bfgamma}^T\bfdelta}<C,
\end{equation}
because $1-4\sqrt{2}\sqrt{\bar{\kappa}}$ is a positive constant by the condition \eqref{rem_unisol_eq2_1}.  Also, the estimate for $c_0$ is similar to \eqref{bounded_7}. Then, \eqref{bilinear_basis_bounds_2} follows from applying these bounds for $\bfc$ and $c_0$ to \eqref{eq4_13}.
\end{proof}

\begin{lemma} [Partition of Unity]
\label{lem:POU}
On every interface element $T \in \mathcal{T}_h^i$, we have
\begin{equation}
\label{eq4_25}
\sum_{i\in \mathcal{I}}\phi^P_{i,T}(X) = 1,~~\forall X \in T,
\end{equation}
\begin{equation}
\label{eq4_26}
\nabla \left( \sum_{i\in\mathcal{I}}\phi^P_{i,T}(X) \right)=\sum_{i\in\mathcal{I}}\nabla\phi^P_{i,T}(X)=0, ~~\forall X \in T.
\end{equation}
\end{lemma}
\begin{proof}
Let $p(X) = \sum_{i\in \mathcal{I}}\phi^P_{i,T}(X)$ and $q(X) = 1$. Obviously both $p(X)$ and $q(X)$ are in the format of \eqref{eq4_10} and they satisfy \eqref{eq4_8} and
\eqref{eq4_9}. Furthermore, it is easy to verify that $p(M_i) = 1 = q(M_i), i \in \mathcal{I}$. Hence,
\begin{eqnarray*}
\sum_{i\in \mathcal{I}}\phi^P_{i,T}(X) = p(X) = q(X) = 1
\end{eqnarray*}
according to the unisolvence stated in Theorem \ref{thm4_1}. Property \eqref{eq4_26} follows from \eqref{eq4_25} directly.
\end{proof}

Now, on every $T \in \mathcal{T}_h^i$, choosing arbitrary points $\overline{X}_i \in l, i \in \mathcal{I}$, we can
construct two vector functions as follows:
\begin{subequations}
\label{eq4_27}
\begin{equation}\label{eq4_27a}
\bfLambda_1(X)=\sum_{i\in\mathcal{I}}(M_i-X)\phi^{P\;+}_{i,T}(X)+\sum_{i\in \mathcal{I}^-}(\overline{M}^+(F) - I)^T(M_i-\overline{X}_i)\phi^{P\;+}_{i,T}(X), \;\;\; \text{if}\; X\in T^+,
\end{equation}
\begin{equation}\label{eq4_27b}
\bfLambda_2(X)=\sum_{i\in\mathcal{I}}(M_i-X)\phi^{P\;-}_{i,T}(X)+\sum_{i\in \mathcal{I}^+}(\overline{M}^-(F) - I)^T(M_i-\overline{X}_i)\phi^{P\;-}_{i,T}(X), \;\;\; \text{if}\; X\in T^-.
\end{equation}
\end{subequations}
It follows from the Lemma \ref{lem2_5} that the two functions in \eqref{eq4_27} are independent of the location of $\overline{X}_i\in l, i \in \mathcal{I}$. Furthermore, from the partition of unity stated in Lemma \ref{lem:POU}, we have
\begin{subequations}
\label{eq4_27_modified}
\begin{align}
\label{}
   \bfLambda_1(X)&=\sum_{i\in \mathcal{I}}M_i\phi^{P\;+}_{i,T}(X) - X + \sum_{i \in \mathcal{I}^-}(\overline{M}^+(F) -I)^T(M_i-\overline{X}_i)\phi^{P\;+}_{i,T}(X),    \\
   \bfLambda_2(X)&=\sum_{i\in \mathcal{I}}M_i\phi^{P\;-}_{i,T}(X) - X + \sum_{i \in \mathcal{I}^+}(\overline{M}^-(F) -I)^T(M_i-\overline{X}_i)\phi^{P\;-}_{i,T}(X),
\end{align}
\end{subequations}
which imply that each component of $\bfLambda_1(X)$ and $\bfLambda_2(X)$ is a polynomial in $\Pi_T$ because $\phi^{P\;s}_{i,T}(X)\in \Pi_T, s = \pm$, for $i\in\mathcal{I}$.
We consider two auxiliary vector functions
\begin{equation}
\begin{split}
\label{eq4_29}
&\bfLambda^+(X)=\bfLambda_1(X), \;\; \text{and} \;\;\;\; \bfLambda^-(X)=(\overline{M}^+(F))^T\bfLambda_2(X).
\end{split}
\end{equation}
Let $d(\bfLambda^s), s = \pm$ be the vector of the coefficients of the second degree term in each component of $\bfLambda^s$.
\begin{lemma}
\label{lem4_2}
$\bfLambda^+$ and $\bfLambda^-$ are such that $d(\bfLambda^+)=d(\bfLambda^-)$.
\end{lemma}
\begin{proof}
Let $d(\phi^{P\; +}_{i,T})=d(\phi^{P\; -}_{i,T})=d_i$, $i\in \mathcal{I}$. By the partition of unity, $\sum_{i\in\mathcal{I}}d_i=0$.
By using \eqref{eq2_10} given in Lemma \ref{lem2_4} and using Lemma \ref{lem2_5} to interchange $\overline{X}_i, i \in \mathcal{I}$ with a fixed $\overline{X} \in l$ and, we have
\begin{equation*}
\begin{split}
\label{}
d(\bfLambda^-)&=\sum_{i\in\mathcal{I}}(\overline{M}^+(F))^TM_id_i+\sum_{i\in \mathcal{I}^+}(I-\overline{M}^+(F))^T(M_i-\overline{X})d_i\\
=&\sum_{i\in\mathcal{I}^-}(\overline{M}^+(F))^TM_id_i+\sum_{i\in\mathcal{I}^+}M_id_i - (I-\overline{M}^+(F))^T\overline{X}\sum_{i\in \mathcal{I}^+}d_i\\
=&\sum_{i\in\mathcal{I}}M_id_i+\sum_{i\in\mathcal{I}^-}(\overline{M}^+(F) - I)^TM_i d_i
+ (I-\overline{M}^+(F))^T\overline{X}\sum_{i\in \mathcal{I}^-}d_i \\
=& \sum_{i\in\mathcal{I}}M_id_i+\sum_{i\in\mathcal{I}^-}(\overline{M}^+(F) - I)^T (M_i - \overline{X}_i) d_i
\end{split}
\end{equation*}
which is exactly $d(\bfLambda^+)$.
\end{proof}

\begin{lemma}
\label{lem4_3}
$\bfLambda^+$ and $\bfLambda^-$ satisfy the condition \eqref{eq4_8}, i.e., $\bfLambda^+(X)|_l=\bfLambda^-(X)|_l$.
\end{lemma}
\begin{proof}
Since $\bfLambda^s(X), s = \pm$ are independent of points $\overline{X}_i \in l, i \in \mathcal{I}$, we can replace these points by an arbitrary $\overline{X} \in l$ without
changing values of $\bfLambda^s(X), s = \pm$. Then, since $\phi^P_{i,T}$ satisfies \eqref{eq4_8}, applying \eqref{eq2_10}, we have
\begin{equation*}
\begin{split}
\bfLambda^+(\overline{X})=&\sum_{i\in\mathcal{I}}(M_i-\overline{X})\phi^{P\;-}_{i,T}(\overline{X})+\sum_{i\in\mathcal{I}^-}(\overline{M}^+(F)-I)^T(M_i-\overline{X})\phi^{P\;-}_{i,T}(\overline{X})\\
=&\sum_{i\in\mathcal{I}^+}(M_i-\overline{X})\phi^{P\;-}_{i,T}(\overline{X}) + (\overline{M}^+(F))^T\sum_{i\in\mathcal{I}^-}(M_i-\overline{X})\phi^{P\;-}_{i,T}(\overline{X})\\
&= (\overline{M}^+(F))^T \left((\overline{M}^-(F))^T \sum_{i\in\mathcal{I}^+}(M_i-\overline{X})\phi^{P\;-}_{i,T}(\overline{X})+\sum_{i\in\mathcal{I}^-}(M_i-\overline{X})\phi^{P\;-}_{i,T}(\overline{X}) \right)\\
&=(\overline{M}^+(F))^T \left(\sum_{i\in\mathcal{I}}(M_i-\overline{X})\phi^{P\;-}_{i,T}(\overline{X})
+ (\overline{M}^-(F) - I)^T \sum_{i\in\mathcal{I}^+}(M_i-\overline{X})\phi^{P\;-}_{i,T}(\overline{X})\right)\\
\end{split}
\end{equation*}
which is exactly $\bfLambda^-(\overline{X})$.
\end{proof}

\begin{lemma}
\label{lem4_4}
$\bfLambda^+$ and $\bfLambda^-$ satisfy the condition \eqref{eq4_9}, i.e., $\beta^+\nabla\bfLambda^+(F)\cdot\mathbf{ n}(F)=\beta^-\nabla\bfLambda^-(F)\cdot\mathbf{ n}(F)$, where the gradient operator is understood as the gradient on each component.
\end{lemma}
\begin{proof}
Again, Lemma \ref{lem2_5} allows us to exchange $\overline{X}_i, i \in \mathcal{I}$ for an arbitrary point $\overline{X} \in l$ in the discussion below. By \eqref{eq4_9}, \eqref{eq4_26}, \eqref{eq2_10}, and \eqref{eq2_12} we have
\begin{equation*}
\begin{split}
\label{}
& \beta^+ \nabla \bfLambda^+(F)\cdot\mathbf{ n}(F) = \sum_{i\in\mathcal{I}}M_i\beta^-\nabla\phi^{P\;-}_{i,T}(F)\cdot\mathbf{ n}(F) \\
& \hspace{0.5in} + (\overline{M}^+(F) - I)^T\sum_{i\in\mathcal{I}^-}(M_i-\overline{X})\beta^-\nabla\phi^{P\;-}_{i,T}(F)\cdot\mathbf{ n}(F)
- \beta^+\mathbf{ n}(F) \\
=&\sum_{i\in\mathcal{I}^+}M_i\beta^-\nabla\phi^{P\;-}_{i,T}(F)\cdot\mathbf{ n}(F) + \sum_{i\in\mathcal{I}^-}\overline{X}\beta^-\nabla\phi^{P\;-}_{i,T}(F)\cdot\mathbf{ n}(F)  \\
& \hspace{0.5in}+ (\overline{M}^+(F))^T\sum_{i\in\mathcal{I}^-}(M_i-\overline{X})\beta^-\nabla\phi^{P\;-}_{i,T}(F)\cdot\mathbf{ n}(F)
- \beta^+\mathbf{ n}(F)\\
=&\sum_{i\in\mathcal{I}^+}(M_i-\overline{X})\beta^-\nabla\phi^{P\;-}_{i,T}(F)\cdot\mathbf{ n}(F)
+(\overline{M}^+(F))^T\sum_{i\in\mathcal{I}^-}(M_i-\overline{X})\beta^-\nabla\phi^{P\;-}_{i,T}(F)\cdot\mathbf{ n}(F) - \beta^+\mathbf{ n}(F)\\
=& (\overline{M}^+(F))^T \left( (\overline{M}^-(F))^T\sum_{i\in\mathcal{I}^+}(M_i-\overline{X})\beta^-\nabla\phi^{P\;-}_{i,T}(F)\cdot \bfn(F) \right. \\
&\hspace{1in} \left. +\sum_{i\in\mathcal{I}^-}(M_i-\overline{X})\beta^-\nabla\phi^{P\;-}_{i,T}(F)\cdot \bfn(F) - \beta^+(\overline{M}^-(F))^T\cdot \bfn(F) \right)\\
=& (\overline{M}^+(F))^T \left( (\overline{M}^-(F))^T\sum_{i\in\mathcal{I}^+}(M_i-\overline{X})\beta^-\nabla\phi^{P\;-}_{i,T}(F)\cdot \bfn(F) \right. \\
&\hspace{1in} \left. +\sum_{i\in\mathcal{I}^-}(M_i-\overline{X})\beta^-\nabla\phi^{P\;-}_{i,T}(F)\cdot \bfn(F) - \beta^-\cdot \bfn(F) \right)
= \beta^-\nabla\bfLambda^-(F)\cdot\mathbf{ n}(F).
\end{split}
\end{equation*}

\end{proof}

\begin{thm}
\label{thm4_3}
On every interface element $T \in \mathcal{T}_h^i$ we have
\begin{subequations}
\label{eq_thm4_3}
\begin{align}
&\sum_{i\in\mathcal{I}}(M_i-X)\phi^{P\;-}_{i,T}(X)+\sum_{i\in \mathcal{I}^+}(\overline{M}^-(F) - I)^T(M_i-\overline{X}_i)\phi^{P\;-}_{i,T}(X) = \bfLambda_1(X) = 0, ~\forall X\in T^-,
\label{eq_thm4_3_1} \\
&\sum_{i\in\mathcal{I}}(M_i-X)\phi^{P\;+}_{i,T}(X)+\sum_{i\in \mathcal{I}^-}(\overline{M}^+(F) - I)^T(M_i-\overline{X}_i)\phi^{P\;+}_{i,T}(X) = \bfLambda_2(X) = 0, ~\forall X\in T^+,
\label{eq_thm4_3_2}
\end{align}
\end{subequations}
and
\begin{subequations}
\label{eq_thm4_4}
\begin{align}
&\sum_{i\in\mathcal{I}}(M_i-X)\partial_d\phi^{P\;-}_{i,T}(X)+\sum_{i\in \mathcal{I}^+}\left[ (\overline{M}^-(F) - I)^T(M_i-\overline{X}_i)\partial_s\phi^{P\;-}_{i,T}(X) \right] - \mathbf{ e}_d = 0, \forall X \in T^-, \label{eq4_30} \\
&\sum_{i\in\mathcal{I}}(M_i-X)\partial_d\phi^{P\;+}_{i,T}(X)+\sum_{i\in \mathcal{I}^-}\left[ (\overline{M}^+(F) - I)^T(M_i-\overline{X}_i)\partial_s\phi^{P\;+}_{i,T}(X) \right] - \mathbf{ e}_d = 0, \forall X \in T^+, \label{eq4_31}
\end{align}
\end{subequations}
where $d = 1, 2$, $\partial_1 = \partial_x, \partial_2 = \partial_y$ are partial differential operators, and $\mathbf{ e}_d, d = 1, 2$ is the standard $d$-th unit vector in
$\mathbb{R}^2$.

\end{thm}
\begin{proof}
We define a piecewise vector polynomial on $T$ as
\begin{equation*}
\bfLambda(X)=
\left\{\begin{array}{cc}
\bfLambda^+(X) & \text{if} \; X\in T^+, \\
\bfLambda^-(X) & \text{if} \; X\in T^-.
\end{array}\right.
\end{equation*}
First, the restriction of each component of $\bfLambda$ to $T^s, s = \pm$ is a polynomial in $\Pi_T$. By Lemmas \ref{lem4_2}-\ref{lem4_4}, the components of
$\bfLambda$ also satisfy \eqref{eq4_8} and \eqref{eq4_9}. In addition, we can easily see that $\bfLambda(M_i)= {\bf 0}, i\in \mathcal{I}$. Therefore,
by the unisolvence stated in Theorem \ref{thm4_1}, we have $\bfLambda^+(X)=\bfLambda_1(X)\equiv0$ and $\bfLambda^-(X)=(\overline{M}^+(F))^T\bfLambda_2(X)\equiv0$. Since $(\overline{M}^+(F))^T$ is nonsingular, we have $\bfLambda_2(X)\equiv0$. Therefore \eqref{eq_thm4_3} is proved. The proof for \eqref{eq_thm4_4} can be accomplished by differentiating \eqref{eq_thm4_3} and applying \eqref{eq4_27_modified}.
\end{proof}


\subsection{Optimal Approximation Capabilities of IFE Spaces}

As usual, the local IFE spaces on elements can be employed to define the IFE function space globally on $\Omega$. As an example, we consider
\begin{equation}
\begin{split}
\label{eq4_23}
S_h(\Omega)=&\left\{ v\in L^2(\Omega):v|_{T}\in S^{P}_h(T); \right. \\
& \hspace{0.2in} \left. v|_{T_1}(M)=v|_{T_2}(M)~\forall M \in \mathcal{N}_h, \forall\,T_1, T_2 \in \mathcal{T}_h \text{~such that~} M \in T_1\cap T_2 \right\},
\end{split}
\end{equation}
where $\mathcal{N}_h$ is the set of nodes in the mesh $\mathcal{T}_h$, and this implies every IFE
function in this space is continuous at every node in the mesh. However, as observed in \cite{2015LinLinZhang}, every $v\in S_h(\Omega)$ is usually discontinuous across the interface edges. When the interface is a generic curve, the discontinuity of $v\in S_h(\Omega)$ also occurs along the interface curve because these IFE shape functions are defined according to the actual interface. These features ensure the $H^1$ regularity of $v\in S_h(\Omega)$ in the subdomain
of $\Omega$ minus the union of interface elements.

\par
We proceed to show that these IFE spaces formed above by linear, bilinear, CR and rotated-$Q_1$ polynomials have the optimal approximation property from the point of view how
well the interpolation of a function $u$ in these IFE spaces can approximate $u$. First we define the interpolation operator on an element $T \in \mathcal{T}_h$ as the mapping $I_{h,T}: C^0(T) \rightarrow S_h^P(T)$ such that
\begin{equation}
\label{eq4_33}
I^P_{h,T}u(X)= \begin{cases}
\sum_{i\in\mathcal{I}}u(M_i)\psi^P_{i,T}(X), & \text{if~} T \in \mathcal{T}_h^n, \\
\sum_{i\in\mathcal{I}}u(M_i)\phi^P_{i,T}(X), & \text{if~} T \in \mathcal{T}_h^i.
\end{cases}
\end{equation}
Furthermore, the global IFE interpolation $I^P_{h}: C^0(\Omega)\rightarrow S_h(\Omega)$ can be defined piecewisely as
\begin{equation}
\label{eq4_34}
(I^P_hu)|_{T}=I^P_{h,T}u, \;\;\; \forall T\in \mathcal{T}_h.
\end{equation}

On every non-interface element $T\in\mathcal{T}^n_h$, the standard scaling argument \cite{1994BrennerScott,1978Ciarlet,1992RanacherTurek}
yields the following error estimate for the local interpolation $I_{h,T}^Pu$ on $T$:
\begin{equation}
\label{eq4_32}
\| I^P_{h,T}u-u\|_{0,T}+h| I^P_{h,T}u-u|_{1,T}\leqslant Ch^2|u|_{2,T}, ~\forall u \in H^2(T).
\end{equation}
However, how to use the scaling argument to derive an error bound for the interpolation on an interface element is unclear because the local IFE space $S^{P}_h(T)$ is interface dependent and it is not even a subspace of $H^1(T)$ in general.  Instead, we will use the multi-point Taylor expansion method \cite{2009HeTHESIS,2008HeLinLin,2004LiLinLinRogers,1982Xu,2013ZhangTHESIS}
to derive estimates for the IFE interpolation error.

\begin{thm}
\label{thm4_5}
Let $T \in \mathcal{T}_h^i$, assume $u\in PC_{int}^2(T)$. Then for any $\overline{X}_i\in l, i \in \mathcal{I}$, we have
\begin{subequations}\label{eq4_35}
\begin{equation}
\label{eq4_35a}
I^P_{h,T}u(X)-u(X)=\sum_{i\in\mathcal{I}^{s'}}(E_i^s+F_i^s)\phi^P_{i,T}(X)+\sum_{i\in\mathcal{I}}R_i^s\phi^P_{i,T}(X),~\forall X \in T^s, ~s = \pm,
\end{equation}
\begin{equation}
\label{eq4_35b}
\partial_d( I^P_{h,T}u(X)-u(X))=\sum_{i\in\mathcal{I}^{s'}}(E_i^s + F_i^s)\partial_d\phi^P_{i,T}(X)+\sum_{i\in\mathcal{I}}R_i^s \partial_d\phi^P_{i,T}(X), ~\forall X \in T^s, ~s = \pm,
\end{equation}
\end{subequations}
where $d=1$ or $2$, $R_i^s$ are given by \eqref{eq3_2} and \eqref{eq3_4}, and
\begin{equation}
\label{eq4_37}
\begin{split}
&E_i^s = \left( (M^s(\widetilde{Y}_i)-\overline{M}^s(F))\nabla u^s(X) \right)\cdot (M_i-\widetilde{Y}_i), \;\;\; s = \pm, ~i\in\mathcal{I}^{s'}, \\
&F_i^s = -\left( (\overline{M}^s(F) - I)\nabla u^s(X) \right) \cdot (\widetilde{Y}_i-\overline{X}_i), \;\;\; s = \pm, ~i\in\mathcal{I}^{s'}.
\end{split}
\end{equation}
\end{thm}
\begin{proof}
For $X \in T^s, s = \pm$, substituting the expansion \eqref{eq3_1} and \eqref{eq3_3} into the IFE interpolation \eqref{eq4_33}, we have
\begin{equation}
\begin{split}
\label{eq4_38}
I^P_{h,T}u(X) &= u^s(X)\sum_{i\in\mathcal{I}}\phi^P_{i,T}(X)+ \nabla u^s(X)\cdot \sum_{i\in\mathcal{I}}(M_i-X)\phi^P_{i,T}(X)\\
+ &\sum_{i\in \mathcal{I}^{s'}}\left( \left( M^s(\widetilde{Y}_i)-I \right)\nabla u^-(X) \right)\cdot (M_i-\widetilde{Y}_i)\phi^P_{i,T}(X) + \sum_{i\in\mathcal{I}}R_i^s\phi^P_{i,T}(X),
~X \in T^s, ~s = \pm.
\end{split}
\end{equation}
From Theorem \ref{thm4_3}, we have
\begin{equation}
\begin{split}
\label{eq4_39}
\sum_{i\in\mathcal{I}}(M_i-X)\phi^P_{i,T}(X)=-\sum_{i\in \mathcal{I}^{s'}}(\overline{M}^s(F) - I)^T(M_i-\overline{X}_i)\phi^{P}_{i,T}(X), ~\forall X \in T^s, ~s = \pm.
\end{split}
\end{equation}
Then, applying \eqref{eq4_39} and the partition of unity to \eqref{eq4_38} leads to
\begin{equation}
\begin{split}
\label{eq4_40}
I^P_{h,T}u(X) &= u^s(X)- \sum_{i\in\mathcal{I}^{s'}}\left( (\overline{M}^s(F) - I)\nabla u^s(X) \right) \cdot(M_i-\overline{X}_i)\phi^P_{i,T}(X)\\
+& \sum_{i\in \mathcal{I}^{s'}}\left( \left( M^s(\widetilde{Y}_i)-I \right)\nabla u^s(X) \right)\cdot (M_i-\widetilde{Y}_i)\phi^P_{i,T}(X)+\sum_{i\in\mathcal{I}}R_i^s\phi^P_{i,T}(X),
~X \in T^s, ~s = \pm
\end{split}
\end{equation}
from which \eqref{eq4_35a} follows by using $M_i-\overline{X}_i=(M_i- \widetilde{Y}_i)+(\widetilde{Y}_i -\overline{X}_i)$. For \eqref{eq4_35b}, applying the expansions \eqref{eq3_1} and \eqref{eq3_3}
in $\partial_d I^P_{h,T}u(X)=\sum_{i\in\mathcal{I}}u(M_i) \partial_d\phi^P_{i,T}(X), d = 1, 2$, we have
\begin{equation}
\begin{split}
\label{eq4_42}
& \partial_dI^P_{h,T}u(X) = u^s(X)\sum_{i\in\mathcal{I}}\partial_d\phi^P_{i,T}(X)+ \nabla u^s(X)\cdot \sum_{i\in\mathcal{I}}(M_i-X)\partial_d\phi^P_{i,T}(X)\\
&~~+ \sum_{i\in \mathcal{I}^{s'}}\left( \left( M^s(\widetilde{Y}_i)-I \right)\nabla u^s(X)\right) \cdot (M_i-\widetilde{Y}_i)\partial_d\phi^P_{i,T}(X)+\sum_{i\in\mathcal{I}}R_i^s\partial_d\phi^P_{i,T}(X), ~\forall X \in T^s, ~s = \pm.
\end{split}
\end{equation}
Then, applying \eqref{eq4_26} and Theorem \ref{thm4_3} to \eqref{eq4_42} we have
\begin{equation*}
\begin{split}
&\partial_dI^P_{h,T}u(X) = \nabla u^s(X)\cdot \mathbf{ e}_d - \sum_{i\in\mathcal{I}^{s'}}\left( \left(\overline{M}^s(F) - I\right)\nabla u^s(X) \right)\cdot(M_i-\overline{X}_i)\partial_d\phi^P_{i,T}(X)\\
&~~+\sum_{i\in \mathcal{I}^{s'}}\left( \left( M^s(\widetilde{Y}_i)-I \right)\nabla u^s(X)\right)\cdot (M_i-\widetilde{Y}_i)\partial_d\phi^P_{i,T}(X)+\sum_{i\in\mathcal{I}}R_i^s\partial_x\phi^P_{i,T}(X), ~\forall X \in T^s, ~s = \pm \\
\end{split}
\end{equation*}
which leads to \eqref{eq4_35b} because $\nabla u^s(X)\cdot \mathbf{ e}_d=\partial_d u^s(X)$ and $M_i-\overline{X}_i=(M_i-\tilde{Y}_i)+(\tilde{Y}_i-\overline{X}_i)$.
\end{proof}

By an argument similar to that used in \cite{2013ZhangTHESIS}, we can estimate $E_i^s$ and $F_i^s$ in \eqref{eq4_37} by geometric properties established in Section 3.
\begin{lemma}
\label{lem4_5}
There exist constants $C>0$ independent of the interface location and $\beta^{\pm}$ such that the following estimates hold for
every $T\in\mathcal{T}^i_h$ and $u\in PC^2_{int}(T)$:
\begin{subequations}
\begin{equation}
\label{eq_lem4_5_1}
\| E_i^-\|_{0,T^-}\leqslant C\kappa h^2 |u|_{1,T^-}, \;\;\;\; i\in \mathcal{I}^{+}, ~~~~\| E_i^+\|_{0,T^+}\leqslant C\frac{\kappa}{\rho}h^2 |u|_{1,T^+}, \;\;\;\; i\in \mathcal{I}^{-}.
\end{equation}
\begin{equation}
\label{eq_lem4_5_2}
\| F_i^-\|_{0,T^-}\leqslant Ch^2 |u|_{1,T^-}, \;\;\;\; i\in \mathcal{I}^{+}, ~~~~\| F_i^+\|_{0,T^+}\leqslant \frac{C}{\rho}h^2 |u|_{1,T^+}, \;\;\;\; i\in \mathcal{I}^{-}.
\end{equation}
\end{subequations}
\end{lemma}
\begin{proof}
By $\|M_i-\widetilde{Y}_i\|\leqslant h$ and Lemma \ref{lem2_6}, we have
\begin{equation}
\|E_i^s\|_{0,T^s}\leqslant\|(M^s(\widetilde{Y}_i)-\overline{M}^s(F))\| \; \|\nabla u^s\|_{0,T^s} \; \|(M_i-\widetilde{Y}_j)\|
\end{equation}
which implies the \eqref{eq_lem4_5_1} by \eqref{eq2_13} in Lemma \ref{lem2_6}. By Lemma \ref{lem2_5}, $F_i^s$ is independent of the choice of $\overline{X}_i \in l$. Hence,
by taking $\overline{X}_i=\widetilde{Y}_{i\bot}$ in $F_i^s$ and applying Lemma \ref{lem2_3}, we have
\begin{equation}
\|F^s_i\|_{0,T^s}\leqslant\|\overline{M}^s(F) - I\| \; \|\nabla u^s\|_{0,T^s} \; \|\widetilde{Y}_i-\widetilde{Y}_{i\bot}\|
\end{equation}
which establishes \eqref{eq_lem4_5_2} by \eqref{bounds_for_Mbar}.
\end{proof}

Now we are ready to prove the main result in this section.
\begin{thm}
\label{thm4_6}
Assume all the conditions required by Theorem \ref{thm4_2} hold and $u\in PH^2_{int}(T)$. Then on every $T\in\mathcal{T}^i_h$ the following hold.
\begin{itemize}
\item
For rotated-$Q_1$ and Crouzeix-Raviart finite elements,
\begin{subequations}
\label{eq4_47}
\begin{equation}
\label{eq4_47_a}
\| I^P_{h,T}u-u\|_{0,T^-} + h| I^P_{h,T}u-u |_{1,T^-}\leqslant C\frac{1+\kappa}{\rho}h^2 (|u|_{1,T}+|u|_{2,T}),
\end{equation}
\begin{equation}
\label{eq4_47_a}
\| I^P_{h,T}u-u\|_{0,T^+} + h| I^P_{h,T}u-u |_{1,T^+}\leqslant C\frac{1+\kappa}{\rho^2}h^2 (|u|_{1,T}+|u|_{2,T}),
\end{equation}
\end{subequations}
where $C$ depends on $\lambda$ chosen for rotated-$Q_1$ case in \eqref{rem_unisol_eq1_1}.
\item
For linear finite elements,
\begin{subequations}
\label{linear_local_est}
\begin{equation}
\label{linear_local_est_a}
\| I^P_{h,T}u-u\|_{0,T^-} +  h| I^P_{h,T}u-u |_{1,T^-}\leqslant C\frac{1+\kappa}{\rho}h^2 (|u|_{1,T}+|u|_{2,T}),
\end{equation}
\begin{equation}
\label{linear_local_est_b}
\| I^P_{h,T}u-u\|_{0,T^+} + h| I^P_{h,T}u-u |_{1,T^+}\leqslant C\frac{1+\kappa}{\rho^{3/2}}h^2 (|u|_{1,T}+|u|_{2,T}).
\end{equation}
\end{subequations}

\item
For bilinear finite elements,
\begin{equation}
\label{bilinear_local_est}
\| I^P_{h,T}u-u\|_{0,T} +  h| I^P_{h,T}u-u |_{1,T}\leqslant C \frac{1+\kappa}{\rho}h^2 (|u|_{1,T}+|u|_{2,T}).
\end{equation}
\end{itemize}
\end{thm}

\begin{proof}
On each $T \in \mathcal{T}_h^i$, by Theorem \ref{thm4_5}, for every $u\in PC^2_{int}(T)$ we have
\begin{equation}
\label{}
| I^P_{h,T}u-u |_{k,T^s}\leqslant \sum_{i \in\mathcal{I}^{s'}}\left(\|E^s_i\|_{0,T^s}+\|F^s_i\|_{0,T^s} \right)|\phi^P_{i,T}|_{k,\infty,T^s} +\sum_{i\in \mathcal{I}}\|R^s_i\|_{0,T^s}|\phi^P_{i,T}|_{k,\infty,T^s},
\end{equation}
where $k=0,1$.
Then, applying Lemmas \ref{lem3_3}, \ref{lem3_5}, \ref{lem4_5} and Theorem \ref{thm4_2} for corresponding IFE functions to the inequality above,
we obtain \eqref{eq4_47}-\eqref{bilinear_local_est} for $u\in PC_{int}^2(T)$. Finally,
the density hypothesis \textbf{(H4)} shows that \eqref{eq4_47}-\eqref{bilinear_local_est} also hold for any $u\in PH^2_{int}(T)$.
\end{proof}

The local estimate in Theorem \ref{thm4_6} leads to the following global estimate for the IFE interpolation directly.
\begin{thm}
\label{thm5_7}
For any $u\in PH^2_{int}(\Omega)$, the following estimate of interpolation error holds
\begin{equation}
\label{eq4_49}
\| I^P_{h}u-u\|_{0,\Omega^s} + h| I^P_{h}u-u |_{1,\Omega^s}\leqslant \widetilde{C}^sh^2 (|u|_{1,\Omega}+|u|_{2,\Omega}), ~~ s=\pm.
\end{equation}
The constants $\widetilde{C}$ depending on $\kappa$ and $\rho$ are specified as the following:
\begin{itemize}
\item for the rotated-$Q_1$ and Crouzeix-Raviart IFE space,
\begin{equation}
\label{interp_estimates_Q1_CR}
\widetilde{C}^-=C\frac{1+\kappa}{\rho},~~~\widetilde{C}^+=C\frac{1+\kappa}{\rho^2},
\end{equation}
where $C$ depends on $\lambda$ for rotated-$Q_1$ case;
\item for the linear IFE space,
\begin{equation}
\label{interp_estimates_linear}
\widetilde{C}^-=C\frac{1+\kappa}{\rho},~~~\widetilde{C}^+=C\frac{1+\kappa}{\rho^{3/2}};
\end{equation}
\item for the bilinear IFE space,
\begin{equation}
\label{interp_estimates_linear}
\widetilde{C}^-=\widetilde{C}^+=C\frac{1+\kappa}{\rho}.
\end{equation}
\end{itemize}
\begin{proof}
Estimate \eqref{eq4_49} follows directly from combining the estimates \eqref{eq4_47} to \eqref{bilinear_local_est} and \eqref{eq4_32}.
\end{proof}
\end{thm}


\section{Numerical Examples}

In this section we use numerical examples to demonstrate the approximation capability of the IFE spaces by IFE interpolation and IFE solutions. In generating numerical results, all computations involving integrations on interface elements, such as the assemblage
of local matrices and vectors with IFE shape functions or assessing the errors with integral norms, are handled by the numerical quadratures based on
the transfinite mapping between the reference straight edge triangle/square and the physical curved edge triangles/quadrilaterals. More
details about quadrature techniques on curved-edge domains can be found in \cite{2009Kopriva,2011SevillaSonia}.

All the numerical results to be presented are generated in the domain $\Omega = (-1,1)\times (-1,1)$ in which the interface curve $\Gamma$ is a circle with radius $r_0=\pi/6.28$ which divides $\Omega$ into two subdomains $\Omega^-$ and $\Omega^+$ with
$$
\Omega^-=\{ (x,y): x^2+y^2 < r_0^2 \}.
$$
The function to be approximated is
\begin{equation}
\label{example_u}
u(x,y) =\left\{
\begin{aligned}
&\frac{1}{\beta^-}r^{\alpha},  &\; (x,y)\in \Omega^-,\\
&\frac{1}{\beta^+}r^{\alpha}+\left( \frac{1}{\beta^-} -\frac{1}{\beta^+} \right)r_0^{\alpha}, & \; (x,y) \in \Omega^+,
\end{aligned}
\right.
\end{equation}
where $r=\sqrt{x^2+y^2}$ and $\alpha=5$. It is easy to verify that $u$ satisfies the interface jump condition \eqref{eq1_3} and \eqref{eq1_4}. We note that this is the same interface problem for the numerical examples given in \cite{2015LinLinZhang}. Numerical examples presented here are generated with the bilinear IFE space developed in Section \ref{sec:IFE Spaces and Their Properties}, and we note that numerical results with other IFE spaces developed in Section \ref{sec:IFE Spaces and Their Properties} are similar which are therefore not presented in order to avoid redundancy.

Note that the curvature of the interface in this interface problem is $\kappa\approx2$.
Condition \eqref{rem_unisol_eq2_1} allows us to use $\bar{\kappa} \approx 0.031$. Then using $\epsilon \approx 0.4$ in
\eqref{h_small_assump} leads to a suggested bound for the mesh size $h \approx 0.0273$. Therefore, our numerical experiments presented in this
section are all on meshes whose sizes are not larger than $1/40 = 0.025$ which can sufficiently satisfy the conditions in the error estimation in the previous section. \\
\commentout{
clear
kappa_b = 0.031
lambda = 1 - 4*sqrt(2)*sqrt(kappa_b)
kappa = 2
ep = 0.4
[sqrt(kappa_b)/(sqrt(2)*(1 + (1-2*ep^2)^(-3/2)*kappa)), ep/kappa]
1./[20, 40]
}

\noindent
{\bf The convergence of IFE interpolation}: Table \ref{table:bilinearIterpolationError_1_10000} and Table \ref{table:bilinearIterpolationError_10000_1} present interpolation error
$u - I_hu$ in both the $L^2$ and the semi-$H^1$ norms over a sequence of meshes whose mesh size is $h$. In these tables, the rate is the estimated
values of $r$ such that $\|u - I_hu\|_{0,\Omega} \approx Ch^r$ or $|u - I_hu|_{1,\Omega} \approx Ch^r$ with numerical results generated on two consecutive meshes. The estimated
values for $r$ clearly demonstrate the optimal convergence of $I_hu$. We note that this example involves a coefficient $\beta$ with a jump quite large. Our numerical experiments show that these IFE spaces converge optimally also when $\beta$ has a moderate jump such as $\beta^- = 1, \beta^+ = 10$.



\begin{table}[H]
\begin{center}
\begin{tabular}{|c |c c|c c|}
\hline
$h$    & $\|u - I_hu\|_{0,\Omega}$ & rate   & $|u - I_hu|_{1,\Omega}$ & rate   \\ \hline
1/40   & 2.7681E-4                 &        & 1.4482E-2               & \\ \hline
1/80   & 7.2447E-4                 & 1.9339 & 7.4468E-3               & 0.9596 \\ \hline
1/160  & 1.8580E-5                 & 1.9632 & 3.7827E-3               & 0.9772 \\ \hline
1/320  & 4.7122E-6                 & 1.9793 & 1.9061E-3               & 0.9888 \\ \hline
1/640  & 1.1858E-6                 & 1.9906 & 9.5723E-4               & 0.9937 \\ \hline
1/1280 & 2.9744E-7                 & 1.9952 & 4.7965E-4               & 0.9969 \\ \hline
\end{tabular}
\end{center}
\caption{Interpolation errors and rates for the bilinear IFE function, $\beta^-=1$ and $\beta^+=10000$}
\label{table:bilinearIterpolationError_1_10000}
\end{table}

\begin{table}[H]
\begin{center}
\begin{tabular}{|c |c c|c c|}
\hline
$h$    & $\|u - I_hu\|_{0,\Omega}$ & rate   & $|u - I_hu|_{1,\Omega}$ & rate   \\ \hline
1/40   & 9.0663E-3                 &        & 4.3850E-1               &  \\ \hline
1/80   & 2.2680E-3                 & 1.9991 & 2.1939E-1               & 0.9991 \\ \hline
1/160  & 5.6711E-4                 & 1.9997 & 1.0971E-1               & 0.9998 \\ \hline
1/320  & 1.4179E-4                 & 1.9999 & 5.4859E-2               & 0.9999 \\ \hline
1/640  & 3.5447E-5                 & 2.0000 & 2.7430E-2               & 1.0000 \\ \hline
1/1280 & 8.8618E-6                 & 2.0000 & 1.3715E-2               & 1.0000 \\ \hline
\end{tabular}
\end{center}
\caption{Interpolation errors and rates for the bilinear IFE function, $\beta^-=10000$ and $\beta^+=1$}
\label{table:bilinearIterpolationError_10000_1}
\end{table}

\noindent
{\bf The convergence of the IFE solution}: Let $u_h$ be the IFE solution generated by the bilinear IFE space applied in the partially penalized method in
\cite{2015LinLinZhang} for the interface problem \eqref{eq1_1}-\eqref{eq1_4} where $f$ and $g$ are chosen such that $u$ given by \eqref{example_u} is its exact solution. The errors in the bilinear IFE solution generated by the symmetric partially penalized IFE (SPP IFE) method on a sequence of meshes are listed in
Tables \ref{table:bilinearIFESolutionError_SymPen1_10000} and \ref{table:bilinearIFESolutionError_SymPen10000_1}. The values of numerically estimated rate $r$ in these tables clearly indicate the optimal convergence of the bilinear IFE solution gauged in either the $L^2$ norm or $H^1$ norm. We also have carried out extensive numerical experiments by applying the IFE spaces developed in Section
\ref{sec:IFE Spaces and Their Properties} to the partially penalized IFE methods in \cite{2015LinLinZhang} with all the popular penalties, and we have
observed similar optimal convergence in the related IFE solution for this interface problem.

\begin{table}[H]
\begin{center}
\begin{tabular}{| c |c c|c c|}\hline
       & \multicolumn{2}{|c|}{SPP IFE}   & \multicolumn{2}{|c|}{SPP IFE}     \\ \hline
$h$    & $\|u-u_h\|_{0,\Omega}$ & rate   & $|u-u_h|_{1,\Omega}$    & rate    \\ \hline
1/40   & 3.7917E-4              &        & 1.5276E-2               &   \\ \hline
1/80   & 1.0409E-4              & 1.8650 & 7.9599E-3               & 0.9405  \\ \hline
1/160  & 2.5628E-5              & 2.0220 & 3.9096E-3               & 1.0257  \\ \hline
1/320  & 6.6828E-6              & 1.9392 & 1.9501E-3               & 1.0035  \\ \hline
1/640  & 1.7806E-6              & 1.9081 & 9.7745E-4               & 0.9964  \\ \hline
1/1280 & 4.0278E-7              & 2.1443 & 4.8374E-4               & 1.0148  \\ \hline
\end{tabular}
\end{center}
\caption{Errors in the bilinear IFE solution generated by the symmetric partially penalized IFE method, $\beta^-=1$, $\beta^+=10000$.}
\label{table:bilinearIFESolutionError_SymPen1_10000}
\end{table}

\begin{table}[H]
\begin{center}
\begin{tabular}{| c |c c|c c|}\hline
       & \multicolumn{2}{|c|}{SPP IFE}   & \multicolumn{2}{|c|}{SPP IFE}     \\ \hline
$h$    & $\|u-u_h\|_{0,\Omega}$ & rate   & $|u-u_h|_{1,\Omega}$    & rate    \\ \hline
1/40   & 1.0734E-2              &        & 4.4052E-1               &   \\ \hline
1/80   & 2.5715E-3              & 2.0616 & 2.1966E-1               & 1.0040  \\ \hline
1/160  & 6.2918E-4              & 2.0310 & 1.0974E-1               & 1.0012  \\ \hline
1/320  & 1.5709E-4              & 2.0019 & 5.4864E-2               & 1.0001  \\ \hline
1/640  & 4.0137E-5              & 1.9686 & 2.7431E-2               & 1.0000  \\ \hline
1/1280 & 9.8101E-6              & 2.0326 & 1.3715E-2               & 1.0000  \\ \hline
\end{tabular}
\end{center}
\caption{Errors in the bilinear IFE solution generated by the symmetric partially penalized IFE method, $\beta^-=10000$, $\beta^+=1$.}
\label{table:bilinearIFESolutionError_SymPen10000_1}
\end{table}


\end{document}